\documentclass[a4paper,reqno,oneside]{amsart}

\usepackage[T1]{fontenc}
\usepackage[utf8]{inputenc}
\usepackage{lmodern}
\usepackage{microtype}

\usepackage[left=80pt,right=80pt,top=80pt,bottom=94pt,heightrounded]{geometry}

\usepackage{amssymb}
\usepackage{stmaryrd} 
\usepackage{mathtools}

\usepackage{hyperref}
\usepackage[hyperpageref]{backref}
\usepackage[nobysame,alphabetic,initials]{amsrefs}

\DefineSimpleKey{bib}{how}
\DefineSimpleKey{bib}{mrclass}
\DefineSimpleKey{bib}{mrnumber}
\DefineSimpleKey{bib}{fjournal}
\DefineSimpleKey{bib}{mrreviewer}

\renewcommand{\PrintDOI}[1]{%
  \href{http://dx.doi.org/#1}{{\tt DOI:#1}}%
}
\renewcommand{\eprint}[1]{#1}
\BibSpec{book}{%
    +{}  {\PrintPrimary}                {transition}
    +{.} { \PrintDate}                  {date}
    +{.} { \textit}                     {title}
    +{.} { }                            {part}
    +{:} { \textit}                     {subtitle}
    +{,} { \PrintEdition}               {edition}
    +{}  { \PrintEditorsB}              {editor}
    +{,} { \PrintTranslatorsC}          {translator}
    +{,} { \PrintContributions}         {contribution}
    +{,} { }                            {series}
    +{,} { \voltext}                    {volume}
    +{,} { }                            {publisher}
    +{,} { }                            {organization}
    +{,} { }                            {address}
    +{,} { }                            {status}
    +{,} { \PrintDOI}                   {doi}
    +{,} { \PrintISBNs}                 {isbn}
    +{}  { \parenthesize}               {language}
    +{}  { \PrintTranslation}           {translation}
    +{;} { \PrintReprint}               {reprint}
    +{.} { }                            {note}
    +{.} {}                             {transition}
}
\BibSpec{article}{%
    +{}  {\PrintAuthors}                {author}
    +{,} { \textit}                     {title}
    +{.} { }                            {part}
    +{:} { \textit}                     {subtitle}
    +{,} { \PrintContributions}         {contribution}
    +{.} { \PrintPartials}              {partial}
    +{,} { }                            {journal}
    +{}  { \textbf}                     {volume}
    +{}  { \PrintDatePV}                {date}
    +{,} { \issuetext}                  {number}
    +{,} { \eprintpages}                {pages}
    +{,} { }                            {status}
    +{,} { \PrintDOI}                   {doi}
    +{,} { \eprint}        {eprint}
    +{}  { \parenthesize}               {language}
    +{}  { \PrintTranslation}           {translation}
    +{;} { \PrintReprint}               {reprint}
    +{.} { }                            {note}
    +{.} {}                             {transition}
}
\BibSpec{collection.article}{%
    +{}  {\PrintAuthors}                {author}
    +{,} { \textit}                     {title}
    +{.} { }                            {part}
    +{:} { \textit}                     {subtitle}
    +{,} { \PrintContributions}         {contribution}
    +{,} { \PrintConference}            {conference}
    +{}  {\PrintBook}                   {book}
    +{,} { }                            {booktitle}
    +{,} { \PrintDateB}                 {date}
    +{,} { pp.~}                        {pages}
    +{,} { }                            {publisher}
    +{,} { }                            {organization}
    +{,} { }                            {address}
    +{,} { }                            {status}
    +{,} { \PrintDOI}                   {doi}
    +{,} { \eprint}        {eprint}
    +{}  { \parenthesize}               {language}
    +{}  { \PrintTranslation}           {translation}
    +{;} { \PrintReprint}               {reprint}
    +{.} { }                            {note}
    +{.} {}                             {transition}
}
\BibSpec{misc}{%
  +{}{\PrintAuthors}  {author}
  +{,}{ \textit}      {title}
  +{.}{ }             {how}
  +{}{ \parenthesize} {date}
  +{,} { available at \eprint}        {eprint}
  +{,}{ available at \url}{url}
  +{,}{ }             {note}
  +{.}{}              {transition}
}

\theoremstyle{plain}
\newtheorem{theo}{Theorem}[section]%
\newtheorem{prop}[theo]{Proposition}%
\newtheorem{coro}[theo]{Corollary}%
\newtheorem{lemm}[theo]{Lemma}%
\newtheorem{theoremA}{Theorem}

\theoremstyle{definition}%
\newtheorem{defi}[theo]{Definition}%
\theoremstyle{remark}%
\newtheorem{rema}[theo]{Remark}%

\usepackage{tensor}

\usepackage{tikz}
\usetikzlibrary{cd}

\DeclareRobustCommand{\SkipTocEntry}[5]{}

\mathchardef\mhyph="2D 				
\newcommand{\numberset}{\mathbb} 
\newcommand{\N}{\numberset{N}} 
\newcommand{\Z}{\numberset{Z}} 
 
\newcommand{\R}{\numberset{R}}
\newcommand{\bC}{\numberset{C}}

\newcommand{\bH}{\mathbb{H}}

\newcommand{\bL}{\mathbb{L}}
\newcommand{\cC}{\mathcal{C}}
\newcommand{\cE}{\mathcal{E}}

\newcommand{\cI}{\mathcal{I}}
\newcommand{\cK}{\mathcal{K}}
\newcommand{\cL}{\mathcal{L}}
\newcommand{\cM}{\mathcal{M}}
\newcommand{\cN}{\mathcal{N}}

\newcommand{\cP}{\mathcal{P}}

\newcommand{\cS}{\mathcal{S}}
\newcommand{\cT}{\mathcal{T}}
\newcommand{\I}{\mathcal{I}}
\newcommand{\tA}{\tilde A}
\newcommand{\tB}{\tilde B}
\newcommand{\PI}{\langle \cP_\I\rangle}

\newcommand{\Ab}{\mathrm{Ab}}
\newcommand{\alg}{\mathrm{alg}}
\newcommand{\id}{\mathrm{id}}

\newcommand{\Cliff}{\mathrm{Cl}_\bC}
\newcommand{\Cone}{\mathrm{Con}}

\newcommand{\rE}{\mathrm{E}}
\newcommand{\GCalg}{\mathrm{C}^*_G}

\DeclareMathOperator{\Id}{id}

\DeclareMathOperator{\RKK}{\mathcal{R}KK}
\DeclareMathOperator{\RE}{\mathcal{R}E}
\DeclareMathOperator{\KKK}{KK}

\renewcommand{\setminus}{\smallsetminus}

\let\E\relax
\let\Res\relax
\newcommand{\E}{\underline{E}}
\DeclareMathOperator{\Ind}{Ind}

\DeclareMathOperator{\Res}{Res}

\newcommand{\tens}[2]{%
  \mathbin{\tensor*[^#1]{\otimes}{_{#2}}}%
}

\author{Valerio Proietti}
\address{Research Center for Operator Algebras, Department of Mathematics, and Shanghai Key Laboratory of Pure Mathematics and Mathematical Practice, East China Normal University, Shanghai 200241, China}
\email{proiettivalerio@math.ecnu.edu.cn}

\author{Makoto Yamashita}
\address{Department of Mathematics, University of Oslo, P.O. box 1053, Blindern, 0316 Oslo, Norway}
\email{makotoy@math.uio.no}

\date{May 26, 2021 (additional comments); April 23, 2021 (major reorganization); June 14, 2020 (first version)}

\title[Homology and $K$-theory of dynamical systems, I]{Homology and $K$-theory of dynamical systems\\ I. torsion-free ample groupoids}

\begin{document}

\begin{abstract}
Given an ample groupoid, we construct a spectral sequence with groupoid homology with integer coefficients on the second sheet, converging to the $K$-groups of the (reduced) groupoid C$^*$-algebra, provided the groupoid has torsion-free stabilizers and satisfies a strong form of the Baum--Connes conjecture. The construction is based on the triangulated category approach to the Baum--Connes conjecture developed by Meyer and Nest. We also present a few applications to topological dynamics and discuss the HK conjecture of Matui.
\end{abstract}


\maketitle
\setcounter{tocdepth}{1}
\tableofcontents


\section*{Introduction}

In this paper, we look at the $K$-theory of ample Hausdorff groupoids, that is, étale groupoids on totally disconnected spaces, and its relation to groupoid homology.
Such groupoids are closely related to dynamical systems on Cantor sets, such as (sub)shifts of finite type (also called topological Markov shifts) in symbolic dynamics.
While this remains a fundamental example, the second half of the last century saw a rapid development of the theory which resulted in several generalizations involving various geometric, combinatorial, and functional analytic structures.

One important class of Cantor systems comes from minimal homeomorphisms of the Cantor set.
This study was initiated by Giordano, Putnam and Skau \cite{pgs:orbit}, in which they classified minimal homeomorphisms up to orbit equivalence.
Actions of $\Z^k$ on the Cantor set, which are higher rank analogues, also naturally appear from tiling spaces.
More generally, essentially free ample groupoids appear in the study of actions of $\N^k$ by local homeomorphisms on zero-dimensional spaces, where they are known as Deaconu--Renault groupoids \citelist{\cite{deaconu:endo}\cite{exren:semi}}.
This is a convenient framework to understand higher-rank graph C$^*$-algebras.
The étale groupoids, and related invariants such as topological full groups, of such systems proved to be a rich source of interesting examples in the structure theory of discrete groups and operator algebras, see for example \citelist{\cite{monjus:piece}\cite{matui:remcan}\cite{phil:cantorzd}}.

Beyond the theory of dynamical systems, these groupoids also play an important role in the theory of operator algebras, where they provide an invaluable source of examples of C$^*$-algebras.
These are obtained by considering the (reduced) groupoid C$^*$-algebras \cite{ren:group}, generalizing the crossed product algebras for group actions on the Cantor set.
The resulting C$^*$-algebras capture interesting aspects of the homoclinic and heteroclinic structure of expansive dynamics \citelist{\cite{mats:ruellemarkov}\cite{put:algSmale}\cite{thomsen:smale}}, extending the correspondence between topological Markov shifts and the Cuntz--Krieger algebras.

The $K$-groups of groupoid C$^*$-algebras and groupoid cohomology with integer coefficients are known to have close parallels, for example in various cohomological invariants of tiling spaces.
In fact, groupoid homology \cite{cramo:hom} has even closer properties to $K$-groups, and the comparison of these invariants (for topologically free, minimal, and ample Hausdorff groupoids) was recently popularized by Matui \cite{matui:hk}.
While his conjectural isomorphism in its original form (``HK conjecture'') has counterexamples \cite{eduardo:hk}, in situations where one expects low homological dimension we do have an isomorphism, see for example \citelist{\cite{simsfarsi:hk}\cite{ortega:kep}}.

\medskip
Our main result gives a correspondence between groupoid homology and $K$-groups for reduced crossed products by torsion-free ample groupoids satisfying the strong Baum--Connes conjecture \cite{tu:moy}, as follows.

\begin{theoremA}[Theorem \ref{theo:hk-with-coeff}]\label{thm:a}
Let $G$ be an ample groupoid with torsion-free stabilizers satisfying the strong Baum--Connes conjecture, and $A$ be a separable $G$-C$^*$-algebra.
Then there is a convergent spectral sequence
\begin{equation*}
E^2_{pq}  = H_p(G,K_q(A)) \Rightarrow K_{p+q}(G \ltimes A).
\end{equation*}
\end{theoremA}

In particular, for $A = C_0(X)$ we obtain a spectral sequence with $E^2_{pq} = E^3_{pq} = H_p(G, K_q(\bC))$ converging to $K_{p+q}(C^*_r G)$.
Similarly to discrete groups, amenable groupoids satisfy the (strong) Baum--Connes conjecture, which cover most of our concrete examples in this paper.

Note that, for groupoids with low homological dimension, this spectral sequence degenerates for degree reasons.
Moreover the top-degree group in groupoid homology tends to be torsion-free, so that there are no extension problems, thus leading to the positive cases where the HK conjecture holds.

Our proofs of Theorem \ref{thm:a} is based on the triangulated category approach to the Baum--Connes conjecture by Meyer and Nest \citelist{\cite{meyer:tri}\cite{nestmeyer:loc}\cite{meyernest:tri}}.
Building on their theory of projective resolutions and complementary subcategories from homological ideals, we show that an explicit projective resolution can be obtained from adjoint functors and associated simplicial objects.
Applying this to the restriction functor $\KKK^G \to \KKK^X$ and induction functor $\KKK^X \to \KKK^G$ for $X = G^{(0)}$ gives the standard bar complex computing the groupoid homology.
Then, the spectral sequence in Theorem~\ref{thm:a} appears as a particular case of the ``ABC spectral sequence'' of \cite{meyer:tri}.

\medskip
This paper is organized as follows.
In Section \ref{sec:prelim} we lay out the basic notation and definitions for all the background objects of the paper.

In Section \ref{sec:approx-equivar-KK}, we look at a simplicial object arising from adjoint functors and relate it to the categorical approach to the Baum--Connes conjecture.
In a triangulated category, a homological ideal with enough projectives and a pair of complementary subcategories, appear from an adjunction of functors \cite{meyer:tri}.
Our observation is that the canonical comonad construction from homological algebra gives a concrete model of projective resolution.
We then use this to show that, when $G$ is an étale groupoid satisfying the strong Baum--Connes conjecture, any $G$-C$^*$-algebra $A$ belongs to the triangulated (localizing) subcategory of $\KKK^G$ generated by the image of the induction functor $\KKK^X \to \KKK^G$ for $X = G^{(0)}$.

We then combine these results in Section \ref{sec:homology-k-theory} to obtain our main results mentioned above. 
Now, let us summarize the ingredients which go into the correspondence between groupoid homology and $K$-theory.
By the adjunction of the functors $\Ind^G_X \colon \KKK^X \to \KKK^G$ and $\Res^G_X\colon \KKK^G \to \KKK^X$, for any $G$-C$^*$-algebra $A$ we have an exact triangle in $\KKK^G$,
\[
P \to A \to N \to \Sigma P,
\]
with $\Res^G_X N \simeq 0$ and $P$ being orthogonal to all such $N$.
From results of Section \ref{sec:approx-equivar-KK}, for any homological functor $F$, we have a spectral sequence from the Moore complex of the simplicial object $(F(L^{n+1} A))_{n=0}^\infty$ with $L = \Ind^G_X \Res^G_X$, converging to $F(P)$.

In addition, if $G$ has torsion-free stabilizers and satisfies the strong Baum--Connes conjecture, we actually have $P \simeq A$ in $\KKK^G$, hence obtaining a homological computation of $F(A)$.
For an ample groupoid $G$, with the functor $F = K_\bullet(G \ltimes \mhyph)$, this complex is isomorphic to the bar complex computing the groupoid homology of $G$ with coefficient in $K_\bullet(A)$.

Finally, in Section \ref{sec:examples} we discuss some examples.
We also compare our construction with the counterexample to the HK conjecture from \cite{eduardo:hk}.

\addtocontents{toc}{\SkipTocEntry}\subsection*{Acknowledgements}

We are indebted to R.~Nest for proposing the topic of this paper as a research project, and for numerous stimulating conversations.
We are also grateful to R.~Meyer for valuable advice concerning equivariant $K$-theory and for his careful reading of our draft. Thanks also to M.~Dadarlat, R.~Deeley, M.~Goffeng, and I.~F.~Putnam for stimulating conversations and encouragement at various stages, which led to numerous improvements.

This research was partly supported through V.P.'s ``Oberwolfach Leibniz Fellowship'' by the \emph{Mathematisches Forschungsinstitut Oberwolfach} in 2020. In addition, V.P. was supported by the Science and Technology Commission of Shanghai Municipality (grant No.~18dz2271000). M.Y. acknowledges support by Grant for Basic Science Research Projects from The Sumitomo Foundation at an early stage of collaboration.

\section{Preliminaries}\label{sec:prelim}

In this section we recall the most important objects and notions at the basis of this paper. We will deal with C$^*$-algebras endowed with a groupoid action, and will consider these as objects of the equivariant Kasparov category.

\subsection{Groupoids and Morita equivalence}

Let $G$ be a groupoid with base space $X = G^{(0)}$.
We let $s,r\colon G\to X$ denote respectively the source and range maps. 
In addition, we let $G_x=s^{-1}(x)$, $G^x=r^{-1}(x)$, and for a subset $A \subset X$, we write $G_A = \bigcup_{x \in A} G_x$, $G^A = \bigcup_{x \in A} G^x$, and $G|_A = G^A \cap G_A$.

\begin{defi}
A topological groupoid $G$ is \emph{étale} if $s$ and $r$ are local homeomorphisms, and \emph{ample} if it is étale and $G^{(0)}$ is totally disconnected.
\end{defi}

If $G$ is étale and $g \in G$, then by definition, for small enough neighborhoods $U$ of $s(g)$ there is a neighborhood $U'$ of $g$ such that $s(U') = U$, and the restriction of $s$ and $r$ to $U'$ are homeomorphisms onto the images.
When this is the case, we write $g(U) = r(U')$ and use $g$ as a label for the map $U \to g(U)$ induced by the identification of $U \sim U' \sim g(U)$.

Throughout the paper we assume that a topological groupoid is second countable, locally compact Hausdorff, and admits a continuous Haar system $\lambda = (\lambda^x)_{x \in X}$, an invariant continuous distribution of Radon measures on the spaces $(G^x)_{x \in X}$.
In particular, $G$ and $X$ are $\sigma$-compact and paracompact.
Under this setting we have full and reduced groupoid C$^*$-algebras $C^*(G, \lambda)$, $C^*_r(G, \lambda)$ make sense (we mostly focus on the latter).
In general these might depend on the choice of $\lambda$, but different choices lead to strongly Morita equivalent C$^*$-algebras respectively \citelist{\cite{damu:renequiv}\cite{simswill:morita}}.
Recall that the condition on Haar system is automatic for étale groupoids, as we can take the counting measure on $G^x$.
In this case we suppress the notation $\lambda$, and simply write $C^*_r(G)$ instead of $C^*_r(G, \lambda)$.

A locally compact groupoid is \emph{amenable} if there is a net of probability measures on $G^x$ for $x \in G^{(0)}$ which is approximately invariant, see \cite{renroch:amgrp}.
In this case, the full and reduced C$^*$-norms are equal, and the completion of the compactly supported functions in the regular representation is $*$-isomorphic to the full groupoid C$^*$-algebra.

\subsection{Groupoid equivariant \texorpdfstring{C$^*$-algebras}{C*-algebras}}

Let us fix our conventions for $G$-C$^*$-algebras.

\begin{defi}
A \emph{$C_0(X)$-algebra} is a C$^*$-algebra $A$ endowed with a nondegenerate $*$-homomorphism from $C_0(X)$ to the center of the multiplier algebra $\cM(A)$.
\end{defi}

Thus, if $a \in A$, we have $a = f b = b f$ for some $f \in C_0(X)$ and $b \in A$, and the second equality holds for all $f$ and $b$.
For an open set $U \subset X$, we put $A_U = A C_0(U)$.
For a locally closed subset $Y \subset X$, that is, if $Y = U \setminus V$ for some open sets $U, V \subset X$, we put $A_Y = A_U / A_{U \cap V}$, and we put $A_x = A_{\{x\}} = A / A C_0(X \setminus \{x\})$ for $x \in X$.

A $C_0(X)$-algebra is \emph{$C_0(X)$-nuclear} if it is a continuous field of C$^*$-algebras over $X$ such that every fiber $A_x$ is nuclear. There is another way to define this in terms of completely positive maps factoring through $M_n(C_0(X))$, see \cite{bauval:nuc}.

\begin{defi}
Let $A$ and $B$ be $C_0(X)$-algebras which admit faithful $C_0(X)$-equivariant nondegenerate representations on Hilbert C$^*$-$C_0(X)$-modules $\cE$ and $\cE'$.
Then their C$^*$-algebraic relative tensor product $A \otimes_{C_0(X)} B$ is defined as the closure of the image of $A \otimes^{\alg}_{C_0(X)} B$ in the adjointable operators $\cL(\cE \otimes_{C_0(X)} \cE')$.
\end{defi}

Although we do not need it, the above definition can be extended to arbitrary $C_0(X)$-algebras \cite{kas:descent}*{Definition 1.6}.

\begin{rema}
\label{rem:models-of-C0X-tensor-prod}
If $A$ or $B$ is $C_0(X)$-nuclear, we have
\[
A \otimes_{C_0(X)} B \simeq (A \otimes_{\max} B)_{\Delta(X)} \simeq (A \otimes_{\min} B)_{\Delta(X)},
\]
where $\Delta(X) = \{(x, x) \mid x \in X \} \subset X \times X$, see \cite{bla:defhopf}*{Section 3.2}.
\end{rema}

If $f \colon Y \to X$ is a continuous map, $C_0(Y)$ is a $C_0(X)$-algebra. It is  a continuous field (hence $C_0(X)$-nuclear) if and only if $f$ is open \cite{blakirch:glimm}. The map $f$ induces a functor $f^* A = C_0(Y) \otimes_{C_0(X)} A$ from the category of $C_0(X)$-algebras to that of $C_0(Y)$-algebras.
For $Y = G$ and $f = s$, we write $s^* A = C_0(G) \tens{s}{C_0(X)} A$, and similarly for $f = r$.

\begin{defi}
Let $G$ be a second countable locally compact Hausdorff groupoid, and put $G^{(0)} = X$.
A \emph{continuous action} of $G$ on a $C_0(X)$-algebra $A$ is given by an isomorphism of $C_0(G)$-algebras
\[
\alpha\colon C_0(G) \tens{s}{C_0(X)} A \to C_0(G) \tens{r}{C_0(X)} A
\]
such that the induced homomorphisms $\alpha_g \colon A_{s(g)} \to A_{r(g)}$ for $g \in G$ satisfy $\alpha_{g h} = \alpha_g \alpha_h$.
In this case, we say that $A$ is a $G$-C$^*$-\emph{algebra}.
\end{defi}
 
For an étale groupoid $G$, the above amounts to giving $\alpha_g$ as isomorphisms $A_U \to A_{g(U)}$ for small enough neighborhoods $U$ of $s(g)$, compatible with the natural actions of $C_0(U) \simeq C_0(g(U))$ and multiplicative in $g$.

In \cite{gall:kk}, Le Gall constructed the equivariant $\KKK$-category of separable and trivially graded $G$-C$^*$-algebras with morphism sets $\KKK^G(A, B)$, generalizing Kasparov's construction for transformation groupoids.
This will be our main framework to work in.

\begin{rema}
Le Gall uses a different convention for $A \otimes_{C_0(X)} B$, namely $(A \otimes_{\max} B)_{\Delta(X)}$, which is different from ours in general.
However these definitions agree in all the relevant cases, such as $B = C_0(Y)$ for a locally compact space $Y$ endowed with an open map $Y \to X$, then $B$ would be $C_0(X)$-nuclear, see Remark \ref{rem:models-of-C0X-tensor-prod}.
For example, the range map $G\to X$ is open because there exists a Haar system \cite{ren:group}*{Proposition 2.4}.
\end{rema}

Suppose moreover that $G$ admits a Haar system $\lambda$.
The algebraic balanced tensor product $C_c(G) \tens{s}{C_0(X)} A$ admits an $A$-valued inner product induced by the measures on the sets $G_x$ from $\lambda$, and we denote its closure as a right Hilbert $A$-module by $E^G_A = L^2(G, A)$.
(This can be interpreted as $L^2(G) \otimes_{C_0(X)} A$, where the canonical right $C_0(X)$-Hilbert module $L^2(G) = L^2(G, C_0(X))$.)
The \emph{reduced crossed product} $G \ltimes_\alpha A = C^*_r(A, G, \alpha, \lambda)$ is the C$^*$-algebra generated by the ``convolution product'' representation of $C_c(G) \tens{s}{C_0(X)} A$ on $E^G_A$, see \citelist{\cite{skan:crossinv}\cite{damu:renequiv}} for the details.
In this paper we always take reduced crossed products, although they will be isomorphic to the full ones in most of our concrete examples as we mostly consider amenable groupoids.

\begin{rema}
Different choices of $\lambda$ give Morita equivalent reduced crossed products.
More generally, let $H$ be another topological groupoid, and let $(A,G,\alpha)$ and $(B,H,\beta)$ be equivalent actions of $G$ and $H$ in the sense of \cite{damu:renequiv}.
As part of the data there is a bibundle $Z$ over $G$ and $H$ as above, and we get a \emph{linking groupoid} $L = G \coprod Z \coprod Z^\mathrm{op} \coprod H$ with base $G^{(0)} \coprod H^{(0)}$.
On one hand, we have a continuous action $\gamma$ of $L$ on $A \oplus B$ induced by the equivalence data.
On the other, a choice of Haar systems $\lambda$ on $G$ and $\mu$ on $H$ gives a Haar system $\kappa$ on $L$ \cite{simswill:morita}.
Then the reduced crossed product $C^*_r(A \oplus B, L, \gamma, \kappa)$ is a linking algebra between $C^*_r(A, G,\alpha, \lambda)$ and $C^*_r(B, H,\beta, \mu)$.
\end{rema}

\subsection{Equivariant sheaves over ample groupoids}
\label{sec:equivar-sheaves-over-ample-grpd}

Let $G$ be an étale groupoid. The nerve $(G^{(n)})_{n=0}^\infty$ of $G$ form a simplicial space, with the face maps are given by
\[
d^n_i\colon G^{(n)} \to G^{(n-1)}, \quad
(g_{1},. . ., g_{n}) \mapsto \begin{cases}
(g_{2}, . . . , g_{n} ) & \text{if $i=0$} \\
(g_{1}, . . . , g_{i}g_{i+1}, . . . , g_{n} ) & \text{if $1 \leq i \leq n-1$} \\
(g_{1}, ... , g_{n-1} ) & \text{if $i=n$,}
\end{cases}
\]
with $d^1_1 = r$ and $d^1_0 = s$ as maps $G \to X$, while the degeneracy maps are given by insertion of units.
These structure maps are étale maps.

Suppose further that $G$ be an ample groupoid, and $C$ be a commutative group.
For a topological space $Y$, we denote the group of compactly supported continuous functions from $Y$ to $C$ by $C_c(Y, C)$.
The \emph{groupoid homology} of $G$ with coefficients in $C$, denoted $H_\bullet(G, C)$, is the homology of the chain complex $(C_c(G^{(n)},C))_{n = 0}^\infty$ with differential
\begin{align*}
\partial_n &=\sum_{i=0}^n(-1)^i (d^n_i)_* \colon C_c(G^{(n)}, C) \to C_c(G^{(n-1)}, C),&
(d^n_i)_*(f)(x) &= \sum_{d^n_i(y) = x} f(y).
\end{align*}
(This is well defined as $d^n_i$ is étale.)

This is a special case of groupoid homology with coefficients in equivariant sheaves \cite{cramo:hom}.
Let us quickly review this more general setting.
When $G$ is a topological groupoid with base space $X$, a \emph{$G$-equivariant sheaf} (of commutative groups) over $X$ is a sheaf (of commutative groups) $F$ over $X$, together with a morphism $s^* F \to r^* F$ of sheaves over $G$, with analogous multiplicativity conditions to the case of $G$-C$^*$-algebras.

In fact, when $G$ is ample, such $G$-sheaves are represented by \emph{unitary $C_c(G, \Z)$-modules} \cite{MR3270778}.
Here, we consider the convolution product on $C_c(G, \Z)$, and a module $M$ over $C_c(G, \Z)$ is said to be unitary if it has the factorization property $C_c(G, \Z) M = M$.
The correspondence is given by $\Gamma_c(U, F) = C_c(U, \Z) M$ for compact open sets $U \subset X$ if $F$ is the sheaf corresponding to such a module~$M$.

A sheaf $F$ on a topological space $Y$ is called \emph{soft} if, for any closed subspace $K \subset Y$ and $s \in \Gamma(K, F)$, there is a global section $s' \in \Gamma(Y, F)$ such that $s'|_K = s$.
When $Y$ is second countable locally compact and Hausdorff, this is equivalent to \emph{c-softness}, where in the above $K$ is moreover assumed to be compact.

\begin{prop}
\label{prop:sheaf-is-soft}
Let $Y$ be a totally disconnected, second countable, locally compact Hausdorff space.
Then any sheaf of commutative groups on $Y$ is soft.
\end{prop}

This seems to be folklore, but can be obtained as follows.
As $Y$ is locally compact Hausdorff and totally disconnected, each point has a base of neighborhood consisting of compact open sets.
Thus, fixing a point $y$ and its compact open neighborhood $U$, any closed subset of $U$, being compact, also has a base of neighborhoods consisting of compact open subsets of $Y$.
This, combined with the paracompactness of $Y$, implies the (c-)softness of sheaves \cite{MR0345092}*{Sections II.3.3 and II.3.4}.

Back to equivariant sheaves over (second countable) ample groupoids, with $G$, $F$, and $M$ as above, the \emph{homology of $G$ with coefficient in $F$}, denoted $H_\bullet(G, F)$, is the homology of the chain complex $(C_c(G^{(n)},\Z) \otimes_{C_c(X, \Z)} M)_{n=0}^\infty$ with differentials coming from the simplicial structure as above. Concretely, the differential is given by
\begin{align*}
\partial_n&\colon C_c(G^{(n)},\Z) \otimes_{C_c(X, \Z)} M \to C_c(G^{(n-1)},\Z) \otimes_{C_c(X, \Z)} M\\
\partial_n(f \otimes m) &= \sum_{i = 0}^{n-1} (-1)^i (d^n_i)_* f \otimes m + (-1)^n \alpha_n(f \otimes m),
\end{align*}
where $\alpha_n$ is the concatenation of the last leg of $C_c(G^{(n)},\Z)$ with $M$ induced by the module structure map $C_c(G, \Z) \otimes M \to M$.
This definition agrees with the one given in \cite{cramo:hom} as there is no need to take c-soft resolutions of equivariant sheaves by Proposition \ref{prop:sheaf-is-soft}.

More generally, if $F_\bullet$ is a chain complex of $G$-sheaves modeled by a chain complex of unitary $C_c(G,\Z)$-modules $M_\bullet$, we define $\bH_\bullet(G, F_\bullet)$, the \emph{hyperhomology} with coefficient $F_\bullet$, as the homology of the double complex $(C_c(G^{(p)},\Z) \otimes_{C_c(X, \Z)} M_q)_{p,q}$.

As usual, a chain map of complexes of $G$-sheaves $f\colon F_\bullet \to F'_\bullet$ is a \emph{quasi-isomorphism} if it induces quasi-isomorphisms on the stalks.
When $F_\bullet$ and $F'_\bullet$ are bounded from below, such maps induce an isomorphism of the hyperhomology \cite{cramo:hom}*{Lemma 3.2}.

\subsection{Triangulated categorical structures}\label{sec:cploc}\label{subsec:homds}

The framework of triangulated categories is ideal for extending basic constructions from homotopy theory to categories of C$^*$-algebras. Much work in this direction has been carried out by Meyer and Nest in \citelist{\cite{meyer:tri}\cite{nestmeyer:loc}\cite{meyernest:tri}}.

We follow their convention which we quickly recall here.
The fundamental axiom requires that there is an autoequivalence $\Sigma$, and any morphism $f\colon A \to B$ should be part of an exact triangle:
\[
A \to B \to C \to \Sigma A.
\]
An additive functor $F$ between triangulated categories is said to be exact when it intertwines suspensions and preserves exact triangles.

We say that $\cT$ has \emph{countable direct sums} if, given a sequence of objects $(A_n)_{n=1}^\infty$ in $\cT$, there is an object $\bigoplus_{n=1}^\infty A_n$ such that
\[
\cT\left(\bigoplus_{n=1}^\infty A_n, B\right) \simeq \prod_{n=1}^\infty \cT(A_n, B)
\]
naturally in the $A_n$ and $B$.
An exact functor $F$ is \emph{compatible with direct sums} if  it commutes with countable direct sums (see \cite{meyer:tri}*{Proposition 3.14}).

As before let $G$ be a second countable locally compact Hausdorff groupoid with a Haar system.
Note that triangulated categories involving $\KKK$-theory have no more than countable direct sums, because separability assumptions are needed for certain analytical results in the background.

\begin{prop}[\cite{proietti-phd-thesis}*{Section A.3}]
The equivariant Kasparov category $\KKK^G$ is triangulated.
\end{prop}

Here, the suspension functor $\Sigma$ is given by $\Sigma A = C_0(\R,A)$.
Note that Bott periodicity implies $\Sigma^2\simeq\text{id}$, so that $\Sigma$ is also a model of $\Sigma^{-1}$.
The exact triangles are defined as the triangles isomorphic to mapping cone triangles for equivariant $*$-homomorphisms.
See Section \ref{sec:app-triagl-str-equiv-KK} for some details.

We also note that functors such as $A \mapsto G \ltimes A$ and $A \mapsto D \otimes A$ preserve mapping cones, hence define triangulated functors into appropriated (equivariant) $\KKK$-categories.
These are also compatible with countable direct sums.

\smallskip
We call a subcategory \emph{thick} when it is closed under direct summands.

\begin{defi}\label{def:compl-pair}
We call a pair $(\cL,\cN)$ of thick triangulated subcategories of $\cT$ \emph{complementary} if $\cT(P,N)=0$ for all $P\in \cL,N\in\cN$, and for any $A\in \cT$, there is an exact triangle
\[
P_A \to A \to N_A \to \Sigma P_A
\]
where $P_A \in \cL$ and $N_A \in \cN$.
\end{defi}

Let us list some of the basic properties of a pair of complementary subcategories (see \cite{nestmeyer:loc}*{Proposition 2.9}).
\begin{itemize}
\item We have $N\in \cN$ if and only if $\cT(P,N)=0$ for all $P\in \cL$. Analogously, we have $P \in \cL$ if and only if $\cT(P,N)=0$ for all $N\in \cN$.
\item The exact triangle as above, with $P_A \in \cL$ and $N_A \in \cN$, is uniquely determined up to isomorphism and depends functorially on $A$. In particular, its entries define functors
\begin{align*}
P\colon \cT & \to \cL, \quad A \mapsto P_A, &
N\colon \cT & \to \cN, \quad A \mapsto N_A.
\end{align*}
The functors $P$ and $N$ are respectively left adjoint to the embedding functor $\cP\to \cT$ and right adjoint to the embedding functor $\cN \to \cT$.
\item The \emph{localizations} $\cT/\cN$ and $\cT/\cL$ exist and the compositions
\begin{align*}
\cL &\longrightarrow \cT \longrightarrow  \cT/\cN,&
\cN &\longrightarrow \cT \longrightarrow  \cT/\cL
\end{align*}
are equivalences of triangulated categories.
\end{itemize}

Most concrete examples come from \emph{homological ideals with enough projectives}, as we quickly recall here.
Let $\cT$ and $\cS$ be triangulated categories with countable direct sums, and $F\colon \cT \to \cS$ be an exact functor compatible with direct sums.
The system of morphisms
\[
\cI(A, B) = \ker (F\colon \cT(A, B) \to \cS(F A, F B))
\]
is an example of \emph{homological ideal} compatible with countable direct sums.

\begin{rema}
\label{rem:target-of-homological-functor}
We do not lose generality by assuming that $\cS$ is a stable abelian category, and that $F$ is a stable functor, see \cite{meyernest:tri}*{Remark 19}.
More concretely, we can always replace the target triangulated category $\cS$ by the category of \emph{representable} contravariant functors $\cS \to \Ab$, which are cokernels of the natural transforms $\cS(\mhyph, A) \to \cS(\mhyph, B)$ induced by some $f\colon A \to B$.
\end{rema}

An object $P\in \cT$ is called \emph{$\I$-projective} if $\I(P,A)=0$ for all objects $A\in \cT$.
An object $N\in \cT$ is called \emph{$\I$-contractible} if $\text{id}_N$ belongs to $\I(N,N)$.
Let $\cP_\I, \cN_\I \subseteq\cT$ be the full subcategories of projective and contractible objects, respectively. 
We say that $\cI$ \emph{has enough projectives} if for any $A \in \cT$, there is an $\cI$-projective object $P$ and a morphism $P \to A$ such that, in the associated exact triangle
\[
P \to A \to C \to \Sigma P,
\]
the morphism $A \to C$ belongs to $\cI$.
With $\cI = \ker F$ as above, the latter condition is equivalent to $F P \to  F A$ being a split surjection for all $A$.

We denote by $\langle {P_\I}\rangle$ the ($\aleph_0$-)\emph{localizing} subcategory generated by the projective objects, i.e., the smallest triangulated subcategory that is closed under countable direct sums and contains $P_\I$.
In particular, $\langle P_\I\rangle$ is closed under isomorphisms, suspensions, and if
\[
A \to B \to C \to \Sigma A
\]
is an exact triangle in $\cT$ where any two of the objects $A,B,C$ are in $\langle P_\I\rangle$, so is the third.
Note that $N_\I$ is localizing, and any localizing subcategory is thick. 

\begin{theo}[{\cite{meyer:tri}*{Theorem 3.16}}]\label{thm:pocs}
Let $\cT$ be a triangulated category with countable direct sums, and let $\I$ be a homological ideal with enough projective objects. Suppose that $\I$ is compatible with countable direct sums. Then the pair of localizing subcategories $(\PI,\cN_\cI)$ in $\cT$ is complementary. 
\end{theo}

\begin{rema}
Note that if $(\cL, \cN)$ is a complementary pair, then $\ker P$ has enough projectives and we have $\cL = \cP_{\ker P}$, $\cN = \cN_{\ker P}$.
Thus the above construction is universal, although $\cI$ is not uniquely determined from $(\PI, \cN_\cI)$.
\end{rema}

\begin{defi}
\label{def:I-proj-res}
Let $F\colon \cT \to \cS$ be an exact functor compatible with countable direct sums.
Given an object $A\in \cT$ and a chain complex
\begin{equation}\label{eq:projres}
\begin{tikzcd}
\cdots \ar[r,"\delta_{n+1}"] & P_n \ar[r,"\delta_n"] & \cdots \ar[r,"\delta_1"] & P_0 \ar[r,"\delta_0"] & A,
\end{tikzcd}
\end{equation} 
we say that \eqref{eq:projres} is an (even) \emph{$\cI$-projective resolution} of $A$ if each $P_n$ is $\cI$-projective and the chain complex
\[
\begin{tikzcd}
\cdots\ar[r,"F(\delta_2)"] & F(P_1) \ar[r,"F(\delta_1)"] & F(P_0) \ar[r,"F(\delta_0)"] & F(A) \ar[r] & 0
\end{tikzcd}
\]
is split exact, i.e., is contractible by chain homotopy in $\cS$.
\end{defi}

There is also an intrinsic formulation of \emph{$\cI$-exactness} for chain complexes corresponding to the second condition above, and the above definition does not depend on the choice of $F$ with $\cI = \ker F$.
Moreover, if $\cI$ has enough projectives, any $A$ has an $\cI$-projective resolution.
In particular, two $\cI$-projective resolutions of $A$ are chain homotopy equivalent, and we obtain functor $\cT \to \mathrm{Ho}(\cT)$.
Moreover, if $P_\bullet$ is an $\cI$-projective resolution of $A$, the object $P_A$ in Definition \ref{def:compl-pair} (an \emph{$\cI$-simplicial approximation} of $A$) belongs to the localizing subcategory generated by the objects $P_k$.

\begin{defi}
An \emph{odd} $\cI$-projective resolution is an $\cI$-projective resolution where the boundary maps of positive index have degree one, i.e., the morphism $\delta_n\colon P_n\to P_{n-1}$ gets replaced, for $n\geq 1$, by a morphism $\delta_n\colon P_n \to \Sigma P_{n-1}$.
\end{defi}

Evidently, if $(P_n,\delta_n)$ is an odd projective resolution, then $(P^\prime_n,\delta^\prime_n)$ is an even resolution, where $P^\prime_n=\Sigma^{-n}P_n$ and $\delta^\prime_n=\Sigma^{-n}\delta_n$.

Let $K\colon \cT \to \mathcal{C}$ be a covariant homological functor into a stable abelian category.
We put $K_n(A) =K(\Sigma^{-n}A)$.
Let us recall a few extra constructions on $K$ motivated by homological algebra.

\begin{defi}
Let $(\cL, \cN)$ be a complementary pair, with $P \colon \cT \to \cL$.
The \emph{localization} of $K$ with respect to $\cN$ is defined by $\bL^\cN K=K\circ P$.
\end{defi}

The defining morphisms $P(A)\to A$ induce a natural transformation $\bL^\cN K \Rightarrow K$.

\begin{defi}
Let $\cI$ be a homological ideal with countable direct sums and enough projectives.
The $p$-th \emph{derived functor} of $K$ with respect to $\cI$ is defined as
\[
\bL^\cI_p K(A) = H_p(K(P_\bullet)),
\]
where $P_\bullet$ is any $\cI$-projective resolution of $A$.
\end{defi}

This is well-defined because projective resolutions are unique up to chain homotopy.
Note that unless $K$ is compatible with $\cI$-exact sequences, one cannot expect $\bL^\cI_0 K \simeq K$.
When $(\cL, \cN)$ is a complementary pair, we can think of the localization $\bL^\cN K$ as the derived functor $\bL_0^{\ker P} K$ for $P \colon \cT \to \cL$ up to the embedding of Remark \ref{rem:target-of-homological-functor}.

Building on the idea of Christensen \cite{chri:tri} to understand the Adams spectral sequence, Meyer constructed the following spectral sequence.

\begin{theo}[{\cite{meyer:tri}*{Theorems 4.3 and 5.1}}]\label{thm:spseqtri}
Let $\cI$ be a homological ideal with countable direct sums and enough projectives, and let $K\colon \cT \to \Ab$ be a homological functor to the category of commutative groups.
Then there is a convergent spectral sequence
\[
E^r_{p q} \Rightarrow \bL^{\cN_\cI} K_{p+q}(A),
\]
with the $E^2$-sheet $E^2_{p q} = \bL^\cI_p K_q(A)$.
\end{theo}

The $E^r$-differentials $d^r\colon E^r_{p q} \to E^r_{p-r, q+r-1}$ come from a choice of \emph{phantom tower} for $A$ and the associated \emph{exact couple}, but their precise form will not be important for us.

\subsection{The Baum--Connes conjecture for groupoids}\label{sec:bc}

Because we are particularly interested in spectral sequences which approximate the $K$-groups of groupoid C$^*$-algebras, the Baum--Connes conjecture naturally plays a fundamental role. The notion of pair of complementary subcategories introduced earlier allows for a general formulation of this conjecture in terms of localization at the subcategory contractible objects. 

However, as our main focus is on torsion-free amenable groupoids, we will not need the full machinery for our applications, hence we limit ourselves to simply recalling the main positive result concerning the conjecture for groupoids with the \emph{Haagerup property}.
Namely, $G$ is said to have the Haagerup property if it acts properly and isometrically by affine maps on a continuous field of (real) Hilbert spaces, or equivalently, if there is a proper conditionally negative type function on $G$ \cite{bcv:aat}.
Analogously to the case of groups, amenable groupoids have this property.

Suppose $G$ is a second countable, locally compact, Hausdorff groupoid with Haar system.
In the following, the crossed product is understood to be \emph{reduced}.

\begin{defi}\label{def:proper}
A $G$-algebra $A$ is said to be \emph{proper} if there is a locally compact Hausdorff proper $G$-space $Z$ such that $A$ is a $G \ltimes Z$-algebra. 
\end{defi}

Evidently, a commutative $G$-C$^*$-algebra is proper if and only if its spectrum is a proper $G$-space.

\begin{rema}
If $G$ is locally compact, $\sigma$-compact, and Hausdorff, then there is a locally compact, $\sigma$-compact, and Hausdorff model of $\underbar{E}G$, the classifying space for proper actions of $G$; in our case $G$ is second countable hence $\underbar{E}G$ is too \cite{tu:novikov}*{Proposition 6.15}. 
In Definition \ref{def:proper} for a proper $G$-algebra we can always assume $Z$ to be a model of $\underbar{E}G$.
Indeed if $\phi\colon Z\to \underbar{E}G$ is a $G$-equivariant continuous map, then $\phi^*\colon C_0(\underbar{E}G) \to \cM(C_0(Z)) = C_b(Z)$ can be precomposed with the structure map $\Phi\colon C_0(Z) \to Z(\cM(A))$, making $A$ into an $G \ltimes \underbar{E}G$-algebra. 
\end{rema}

We will need the following result proved by J.-L. Tu.

\begin{theo}[{\cite{tu:moy}}]\label{thm:tu}
Suppose that $G$ has the Haagerup property.
Then there exists a proper $G$-space $Z$ with an open surjective structure morphism $Z \to X$, and a $G \ltimes Z$-C$^*$-algebra $P$ which is a continuous field of nuclear C$^*$-algebras over $Z$, and such that $P\simeq C_0(X)$ in $\KKK^G$. 
\end{theo}

As a consequence, if $A$ is a separable $G$-C$^*$-algebra, then we have that $A\otimes_{C_0(X)}P$ is a proper $G$-C$^*$-algebra that is $\KKK^G$-equivalent to $A$.

In this paper, for a general topological groupoid $G$ we say that it satisfies the strong Baum--Connes conjecture if the conclusions of the previous theorem hold. This definition implies the standard version of the conjecture. More precisely, if $D\colon P \to C_0(X)$ is the isomorphism from Theorem \ref{thm:tu}, there is a commutative diagram
\[
\begin{tikzcd}
\varinjlim_{Y \subseteq \E G} \KKK^G(C_0(Y),A) \arrow[r, "\mu^G_A"] \arrow[d] &  K_\bullet(G \ltimes A) \\
K_\bullet(G \ltimes (A\otimes_{C_0(X)}P)) \arrow[ru, "j_G(D\, \widehat{\otimes} \Id_A)" '] &  
\end{tikzcd}
\]
where all arrows are isomorphisms and $Y$ runs over $G$-compact invariant subsets of $\E G$ (\cite{meyereme:dualities}*{Theorem 6.12}, see also \cite{nestmeyer:loc}).
The functor $j_G$ above is the descent morphism of Kasparov \cite{kas:descent} which has been generalized to this context in \citelist{\cite{gall:kk}\cite{laff:kkban}}. 

\section{Approximation in the equivariant KK-category}
\label{sec:approx-equivar-KK}

\subsection{Simplicial approximation from adjoint functors}

One powerful way to check that a homological ideal has enough projectives is to relate it to adjoint functors between triangulated categories.
More precisely, let $\cS$ and $\cT$ be triangulated categories with countable direct sums, and $E \colon \cS \to \cT$ and $F \colon \cT \to \cS$ be exact functors compatible with countable direct sums, with natural isomorphisms
\begin{equation}
\label{eq:adj-functors}
\cS(A, F B) \simeq \cT(E A, B) \quad (A \in \cS, B \in \cT).
\end{equation}
Then $\cI = \ker F$ has enough projectives and the $\cI$-projective objects are retracts of $E A$ for some $A\in \cS$ \cite{meyernest:tri}*{Section 3.6}.

Our next goal is to give an explicit projective resolution in this setting.
In fact, this situation is quite standard in homological algebra: such adjoint functors give a \emph{comonad} $L = E F$ on $\cT$, from which we obtain a simplicial object $(L^{n+1} A)_{n=0}^\infty$ giving a ``resolution'' of $A$ \cite{wei:homalg}*{Section 8.6}.

\begin{prop}\label{prop:simplicial-res-from-adj-fs}
In the above setting, any $A \in \cT$ admits an $\cI$-projective resolution $P_\bullet$ consisting of $P_n = L^{n+1} A$. The pair of subcategories $(\langle E\cS \rangle, \cN_\cI)$ is complementary.
\end{prop}

\begin{proof}
Let us denote the structure morphisms of the adjunction as
\begin{align*}
\epsilon_B &\in \cT(L B, B),&
\eta_A &\in \cS(A, F E A),
\end{align*}
so that the isomorphism \eqref{eq:adj-functors} is given by
\begin{align*}
\cS(A, F B) &\to \cT(E A, B)& \cT(E A, B) &\to \cS(A, F B)\\
 f &\mapsto \epsilon_B E(f)& g &\mapsto F(g) \eta_A.
\end{align*}

As already observed in \cite{meyernest:tri}, the objects of the form $E A$ are $\cI$-projective.
Indeed, if $g \in \cT(E A, B)$ is in the kernel of $F$, the corresponding morphism in $\cS(A, F B)$ is zero by the above presentation.

Next, let us recall the comonad structure on $L$.
There are natural transformations $L \to \Id_\cT$ and $L \to L^2$ defining a coalgebra structure on $L$.
The counit is simply given by the morphisms $\epsilon_B$, while the comultiplication is  given by $\delta_B = E(\eta_{F B}) \in \cT(L B, L^2 B)$.
The compatibility condition between these reduces to consistency between $\epsilon$ and $\eta$.

Now we are ready to define a structure of simplicial object on $(P_n)_{n=0}^\infty$ as in the assertion.
The \emph{face} morphisms $d^n_i \colon P_n \to P_{n-1}$ ($0 \le i \le n$) are
\[
d_i^n = L^i(\epsilon_{L^{n-i} A}) \colon L^{n+1} A \to L^n A,
\]
while the \emph{degeneracy} morphisms $s^n_i\colon P_n \to P_{n+1}$ ($0 \le i \le n$) are
\[
s_i^n=L^i(\delta_{L^{n-i} A})\colon L^{n+1}A \to L^{n+2}A,
\]
see \cite{wei:homalg}*{Paragraph 8.6.4}.
The associated Moore complex on $(P_n)_{n = 0}^\infty$  is given by
\begin{equation}\label{eq:chcom}
\delta_n=\sum_{i=0}^n (-1)^i d^n_i \colon P_n \to P_{n-1},
\end{equation}
together with the augmentation morphism $\delta_0=\epsilon \colon P_0 = L A \to A$.

Let us check the $\cI$-exactness of the augmented complex, or as in Definition \ref{def:I-proj-res}, the split exactness of
\[
\cdots \to F L^2 A \to  F L A \to  F A \to 0
\]
for all $A$ in a natural way.
We claim that the the complex
\[
\cdots \to F L^2 A \to F L A \to F A \to 0
\]
in $\cS$ is contractible.
Again this is a consequence of a standard argument: the contracting homotopy is given by $h_n = \eta_{ F L^n A}\colon F L^n A \to F L^{n+1} A$ for $n \ge 0$, see \cite{wei:homalg}*{Proposition 8.6.10}.

Finally, the assertion that $\langle E\cS \rangle$ and $\cN_\cI$ are complementary follows from Theorem \ref{thm:pocs}.
\end{proof}

We will apply the previous proposition in the setting $\cT=\KKK^G$, $\cS=\KKK^X$, $F=\Res_X^G$.
As the functor $E$, we take
\[
E A = \Ind^G_X A = C_0(G) \tens{s}{C_0(X)} A
\]
for $C_0(X)$-algebras $A$, endowed with the left translation action of $G$.
This is indeed left adjoint to $F$ by \cite{bon:goingdownbc}.

\subsection{The Baum--Connes conjecture for torsion-free groupoids}

Hereafter it is assumed that $G$ is étale and that it satisfies the conclusion of Theorem \ref{thm:tu}.
Our next goal is the following result.

\begin{theo}
\label{thm:KK-induction-from-base}
Suppose that $G$ is an étale groupoid with torsion-free stabilizers satisfying the conclusion of Theorem {\normalfont\ref{thm:tu}}.
Then any $G$-C$^*$-algebra $A$ belongs to the localizing subcategory generated by the $G$-C$^*$-algebras $(\Ind^G_X \Res^G_X)^n A$ for $n \ge 1$.
\end{theo}

The following lemma clarifies the local picture of proper actions. 

\begin{lemm}
\label{lem:prodr}
Let $G$ be an étale groupoid with torsion-free stabilizers, and $G \curvearrowright Z$ a proper action on a locally compact Hausdorff space with the anchor map $\rho\colon Z \to X$.
Then each $z\in Z$ has an open neighborhood $U$ satisfying:
\begin{itemize}
\item $U$ has a compact closure in $Z$;
\item the saturation $G U$ can be identified as $G \times_X U$ as a $G$-space. 
\end{itemize}
\end{lemm}

\begin{proof}
This is essentially contained in Proposition 2.42 of the extended version of \cite{tu:propnonhaus}, but let us give a proof.
First, observe that any $w \in Z$ has trivial stabilizer.
Indeed, on the one hand it can be identified with the inverse image of $(w, w)$ for the action map $\phi\colon G \ltimes Z \to Z \times_X Z$, hence is a compact set by the properness of the action.
On the other hand, it is a subgroup of the stabilizer of $\rho(w)$, which is a torsion-free group, hence it must be trivial.

Next, fix an open neighborhood $V$ of $z$, and put $C = (G \ltimes Z) \setminus V$, where $V $ is identified with an open subset of $G \ltimes Z$ by taking the identity morphisms of $v \in V$.
Since $Z$ is locally compact Hausdorff, $\phi$ is closed (with compact fibers) and in particular $\phi(C)$ is closed in $Z \times_X Z$, and it does not contain $(z, z)$ by the above observation.

Take an open neighborhood $U$ of $z$ such that $U \times_X U$ does not meet $\phi(C)$.
Then the restriction of the action map to $G \times_X U$ is a bijection onto $G U$.
Indeed, if $(g, u)$ and $(g', u')$ had the same image in $G U$, we would have
\[
(u, u') \in U \times_X U \cap \phi(G \ltimes Z) \subset  \phi(V),
\]
which implies $u = u'$ and then $g = g'$.

Finally, as $G \ltimes Z$ is an étale groupoid, the action map $G \times_X U \to Z$ is an open map.
Then we obtain that the bijective continuous map $G \times_X U \to G U$ is a homeomorphism.
\end{proof}

For the next proof we use the \emph{equivariant $E$-theory} of $C_0(Y)$-algebras \cite{MR1789948}.
The equivariant $E$-groups $\rE^Y(A, B)$ (denoted by $\RE(Y; A, B)$ in \cite{MR1789948}) define a triangulated category with countable direct sums and a triangulated functor $\KKK^Y \to \rE^Y$ compatible with countable direct sums.

We are going to use the notion of $\RKK(X)$-nuclearity as defined by Bauval \cite{bauval:nuc}*{Definition 5.1} (see also \cite{skand:knuc}).
Here, we call it $\KKK^X$-nuclearity.
Namely, a $C_0(X)$-algebra $A$ is $\KKK^X$-nuclear if $\id_A \in \KKK^X(A, A)$ is represented by an $X$-$A$-$A$-Kasparov cycle $(\pi, \cE, T)$ such that the left action $\pi\colon A \to \cL(\cE)$ is \emph{strictly $C_0(X)$-nuclear} with respect to the identification $\cL(\cE) = \cM(\cK(\cE))$.

\begin{lemm}
\label{lem:excision-in-rep-KK}
Let $Y$ be a second countable locally compact space, and $(V_k)_{k=0}^\infty$ be a countable and locally finite open covering of $Y$.
If $A$ is a $\KKK^Y$-nuclear $C_0(Y)$-algebra, and if $N$ is a $C_0(Y)$-algebra such that $N_{V_k}$ is $\KKK^{V_k}$-equivalent to $0$ for all $k$, then we have $\KKK^Y(A, N) = 0$.
\end{lemm}

\begin{proof}
By assumption on $A$, we have $\KKK^Y(A, N) \simeq \rE^Y(A, N)$ \cite{MR1789948}*{Theorem 4.7}.
In order to show the latter group vanishes, it is enough to show $\rE^Y(N, N) = 0$.

Put $N_k = N_{V_0 \cup \cdots \cup V_k}$.
We first claim that $\rE^Y(N_k, N) = 0$ for all $k$.
By induction, it is enough to prove this for $k = 1$.
We have an extension of $C_0(Y)$-algebras
\[
0 \to N_0 \to N_1 \to N_{V_1 \cup V_0 \setminus V_0} \to 0.
\]
By assumption $N_0$ is contractible in $\KKK^Y$ (hence in $\rE^Y$).
We also have the contractibility of $N_{V_1 \cup V_0 \setminus V_0}$, as it is a reduction of the $\KKK^{V_1}$-contractible object $N_{V_1}$ to $V_1 \cup V_0 \setminus V_0 = V_1 \setminus V_0$.
Now, the functor $B \mapsto \rE^Y(B, N)$ satisfies excision \cite{MR1789948}*{Theorem 4.17}, which gives an exact sequence of the form
\[
0 = \rE^Y(N_{V_1 \cup V_0 \setminus V_0}, N) \to \rE^Y(N_1, N) \to \rE^Y(N_0, N) = 0,
\]
and we obtain $\rE^Y(N_1, N) = 0$.

The inclusion maps make $(N_k)_{k=0}^\infty$ an inductive system, and $N$ is its inductive limit as a $C_0(Y)$-algebra.
This inductive system is \emph{admissible} in the sense of \cite{nestmeyer:loc}*{Section 2.4} (this condition is automatic for inductive systems in $\rE^Y$, but this example is already admissible in $\KKK^Y$).
In particular, there is an exact triangle of the form
\[
\Sigma N \to \bigoplus_k N_k \to \bigoplus_k N_k \to N.
\]
Since we already have $\rE^Y(\bigoplus_k N_k, N) \simeq \prod_k \rE^Y(N_k, N) = 0$, we obtain $\rE^Y(N, N) = 0$.
\end{proof}

\begin{proof}[Proof of Theorem \ref{thm:KK-induction-from-base}]
Let $A$ be a separable $G$-C$^*$algebra.
The endofunctor $A \otimes_{C_0(X)} -$ on $\KKK^G$ is triangulated, and sends $C_0(X)$ to $A$.
More generally, we have an isomorphism of $G$-C$^*$-algebras between
\[
A \otimes_{C_0(X)} ((\Ind^G_X \Res^G_X)^n C_0(X)) \simeq A \otimes_{C_0(X)}^r C_0(G^{(n)})
\]
and $C_0(G^{(n)}) \tens{s}{C_0(X)} A \simeq (\Ind^G_X \Res^G_X)^n A$ by inductively applying the structure map of $G$-C$^*$-algebra.
Therefore it is sufficient to prove the statement for $A = C_0(X)$.

Let us first show a slightly weaker statement, namely that $C_0(X)$ belongs to the (localizing) triangulated subcategory of $\KKK^G$ generated by objects of the form $\Ind^G_X B$ for $C_0(X)$-algebras $B$. 

By Theorem \ref{thm:tu}, we may replace $C_0(X)$ by a $C_0(Z)$-nuclear $G \ltimes Z$-C$^*$-algebra $C$ for some proper $G$-space $Z$. Let $U \subset Z$ be an open set satisfying the conditions of Lemma \ref{lem:prodr}, and put $V = G U$.
Then the $G$-algebra $C_V = C_0(V) C$ is isomorphic to $\Ind^G_X C_U$.
Indeed, the latter is $C_0(G) \otimes_{C_0(X)} C_U$, and the $G$-equivariant isomorphism $V \simeq G U$ induces $C_V \simeq C_0(G) \otimes_{C_0(X)} C_U$.

Now, take countably many open sets $(U_i)_{i \in I}$ satisfying the conditions of Lemma \ref{lem:prodr}, such that the sets $V_i = G U_i$ cover $Z$ and $(V_i / G)_i$ is a countable and locally finite open cover of $Z/G$ (this is possible by second countability).
We want to say that the (unreduced) ``Čech complex'' of objects $C_{V_{i_1} \cap \cdots \cap V_{i_k}}$ give a resolution of $C$ in $\KKK^{G \ltimes Z}$.
Then, combined with the ``induction functor'' $\KKK^{G \ltimes Z} \to \KKK^G$ (which is really given by the restriction of $C_0(Z)$-algebras to $C_0(X)$-algebras), we get that $C$ is indeed in $\langle \Ind^G_X \KKK^X \rangle$. Suppose $U$ and $U'$ are open sets of $Z$ satisfying the conditions of Lemma \ref{lem:prodr}, and put $V = G U$ and $V' = G U'$.
Then there is an open set $W$ satisfying  the conditions of Lemma \ref{lem:prodr} with $V \cap V' = G W$; indeed, we can take $W = U \cap V'$.
This implies that the $G$-algebras $C_{V_{i_1} \cap \cdots \cap V_{i_k}}$ are all of the form $\Ind^G_X B$.

Now, set $\tilde Z = \coprod_i V_i$, and regard it as a $G \ltimes Z$-space by the canonical equivariant map $\tilde Z \to Z$.
The functors $\Ind^Z_{\tilde Z} \colon \KKK^{G \ltimes \tilde Z} \to \KKK^{G \ltimes Z}$ and $\Res^Z_{\tilde Z} \colon \KKK^{G \ltimes Z} \to \KKK^{G \ltimes \tilde Z}$ make sense.
Concretely, if $B$ is a $G$-equivariant $C_0(Z)$-algebra, we have
\[
\Res^Z_{\tilde Z} B = \bigoplus_i B_{V_i}
\]
endowed with an obvious action of $G$, while for a $G$-equivariant $C_0(\tilde Z)$-algebra $B$, we set $\Ind^Z_{\tilde Z} B$ to be the same C$^*$-algebra as $B$ regarded as a $C_0(Z)$-algebra.
Then we have the standard adjunction
\[
\KKK^{G  \ltimes Z}(\Ind^Z_{\tilde Z}B, B') \simeq \prod_i \KKK^{G \ltimes V_i}(B_{V_i}, B'_{V_i}) \simeq \KKK^{G \ltimes \tilde Z}(B, \Res^Z_{\tilde Z} B').
\]
From this, we see that $L = \Ind^Z_{\tilde Z} \Res^Z_{\tilde Z}$ satisfies
\[
L^k C = \bigoplus_{i_1, \ldots, i_k} C_{V_{i_1} \cap \cdots \cap V_{i_k}}.
\]
By Proposition \ref{prop:simplicial-res-from-adj-fs}, we obtain an exact triangle
\[
P \to C \to N \to \Sigma P
\]
in $\KKK^{G \ltimes Z}$, such that $P$ is in the localizing subcategory generated by objects of the form $\Ind^{G \ltimes Z}_{U_i} B$, and $N \in \ker \Res^{G \ltimes Z}_{G \ltimes \tilde Z}$.

It remains to prove that $N = 0$ in $\KKK^{G \ltimes Z}$. For this it is enough to show that the morphism $C \to N$ in the above triangle is zero. Indeed, $P$ will then be a direct sum of $C$ and $\Sigma N$, but there is no nonzero morphism from $P$ to $\Sigma N$. Since the action of $G$ on $Z$ is free and proper, there is an equivalence of categories between $\KKK^{G \ltimes Z}$ and $\KKK^{Z/G}$, and similar statements hold for the $G$-invariant open sets $V_i$. Under this correspondence, $C$ corresponds to a $\KKK^{Z/G}$-nuclear algebra. Now, Lemma \ref{lem:excision-in-rep-KK} implies that $\KKK^{G \ltimes Z}(C, N) = 0$. We thus know that $C_0(X) \simeq_{\KKK^G} C$ belongs to the localizing subcategory generated by objects of the form $\Ind^G_X B$.

Now, take a distinguished triangle
\[
P' \to C_0(X) \to N' \to \Sigma P'
\]
in $\KKK^G$ corresponding to the complementary pair $(\langle \Ind^G_X \KKK^X \rangle, \cN_{\ker \Res^G_X})$.
On the one hand, since $C_0(X)$ belongs to $\langle \Ind^G_X \KKK^X \rangle$, the morphism $C_0(X) \to N'$ is trivial and we have $P' \simeq C_0(X)$.
On the other hand, since the objects $(\Ind^G_X \Res^G_X)^{n+1}(C_0(X))$ form a $(\ker \Res^G_X)$-projective resolution of $C_0(X)$, the object $P'$ belongs to the triangulated subcategory generated by them.
This proves the assertion.
\end{proof}

\begin{coro}\label{cor:obv}
Let $G$ and $A$ be as in Theorem {\normalfont\ref{thm:KK-induction-from-base}}.
Let $P_X(A)\in \langle \Ind^G_X\KKK^X \rangle$ be the algebra appearing in the exact triangle 
\[
P_X(A) \to A \to N \to \Sigma P_X(A)
\]
that we get by applying Proposition {\normalfont\ref{prop:simplicial-res-from-adj-fs}}. Then we have $P_X(A)\simeq A$ in $\KKK^G$.
Equivalently, we have $\cN_{\ker \Res^G_X} = 0$.
\end{coro}

\begin{coro}\label{cor:specseq}
Let $G$ and $A$ be as in Theorem {\normalfont\ref{thm:KK-induction-from-base}}.
Then we have a convergent spectral sequence
\begin{equation}\label{eq:spseqH}
E^2_{pq}=H_p(K_q(G \ltimes L^{\bullet+1} A)) \Rightarrow K_{p+q}(G \ltimes A),
\end{equation}
where $L^nA= (\Ind_X^G\Res^G_X)^n(A)$.
\end{coro}

\begin{proof}
The reduced crossed product functor
\[
\KKK^G \to \KKK, \quad A \mapsto G \ltimes A
\]
is exact and compatible with direct sums, while
\[
\KKK \to \Ab, \quad B \mapsto K_0(B)
\]
is a homological functor. Thus, their composition
\[
K_0(G \ltimes \mhyph) \colon \KKK^G \to \Ab
\]
is a homological functor, cf.~\cite{meyernest:tri}*{Examples 13 and 15}.
Now we can apply Theorem \ref{thm:spseqtri} to get a spectral sequence
\[
H_p(K_q(G \ltimes P_\bullet)) \Rightarrow K_{p+q}(G \ltimes P_X(A)),
\]
where $P_\bullet$ is a $(\ker \Res^G_X)$-projective resolution of $A$.
The $(\ker \Res^G_X)$-projective resolution from Proposition \ref{prop:simplicial-res-from-adj-fs} gives the left hand side of \eqref{eq:spseqH}. Now the claim follows from Corollary \ref{cor:obv}.
\end{proof}

\section{Homology and \texorpdfstring{$K$}{K}-theory}
\label{sec:homology-k-theory}

Let us assume that $G$ is ample. We are going to relate the results of the previous section to the complex of groupoid homology which was described in Section \ref{sec:equivar-sheaves-over-ample-grpd}. First observe that $L^n A = C_0(G^{(n)}) \otimes A$.

As for the coefficients of homology, for the algebra $C_0(X)$ we have $K_0(C_0(X)) \simeq C_c(X, \Z)$, which corresponds to the constant sheaf $\underline\Z$ on $X$. More generally, any $G$-C$^*$-algebra gives a $G$-sheaf on $X$.

\begin{prop}\label{prop:G-alg-K-grp-CcG-mod}
Let $A$ be a $G$-C$^*$-algebra.
Then $K_i(A)$ is a unitary $C_c(G,\Z)$-module.
\end{prop}

\begin{proof}
We show that $K_i(A)$ is a unitary $C_c(X,\Z)$-module, and the associated sheaf is a $G$-sheaf.

The structure map of $C_0(X)$-algebra induces a $*$-homomorphism $C_0(X) \otimes A \to A$.
Combined with the canonical map $K_0(C_0(X)) \otimes K_i(A) \to K_i(C_0(X) \otimes A)$, we obtain a map $K_0(C_0(X)) \otimes K_i(A) \to K_i(A)$, hence a $C_c(X,\Z)$-module structure on $K_i(A)$.

Next let us check the unitarity of this module.
By total disconnectedness and second countability of $X$, we can take an increasing sequence of compact open sets $(U_k)_{k=1}^\infty$ in $X$ such that $\phi_k = \chi_{U_k}$ form an approximate unit of $C_0(X)$.
Replacing $A$ by its suspension if necessary, it is enough to check that, for any class $c \in K_0(A)$, there is $k$ such that $[\phi_k] \in K_0(C_0(X))$ satisfies $[\phi_k] c = c$.

By definition, $c$ is represented by a formal difference $[e] - [f]$ of projections $e, f \in M_n(A^+)$ such that $\pi([e]) = \pi([f])$ in $K_0(\bC)$, where $n$ is some integer, $A^+$ is the unitization of $A$, and $\pi\colon A^+ \to \bC$ is the canonical quotient map.
By conjugating by a unitary in $M_n(\bC)$, we can arrange $\pi(e) = \pi(f)$.
Then we can write the components of $e$ as $e_{ij} = \alpha_{ij} + e'_{ij}$ for $\alpha_{ij} \in \bC$ and $e'_{ij} \in A$, and those of $f$ as $f_{ij} = \alpha_{ij} + f'_{ij}$.

Now, put $x^{(k)}_{ij} = (1 - \phi_k) \alpha_{ij}$, $y^{(k)}_{ij} = \phi_k e_{ij}$, and $z^{(k)}_{ij} = \phi_k f_{ij}$.
These form projections $x^{(k)} \in M_n(A^+)$, $y^{(k)}, z^{(k)} \in M_n(A)$, such that $x^{(k)} + y^{(k)}$ and $x^{(k)} + z^{(k)}$ are still projections.
If $\phi_k e'_{ij}$ is close enough to $e'_{ij}$ in norm, $e$ is close to $x^{(k)} + y^{(k)}$ in norm, and we obtain $[e] = [x^{(k)} + y^{(k)}]$ in $K_0(A^+)$ for large enough $k$.
Similarly, we obtain $[f] = [x^{(k)} + z^{(k)}]$ for large enough $k$.
Then we have $c = [y^{(k)}] - [z^{(k)}]$, and for this $k$ we indeed have $[\phi_k] c = c$.

It remains to give an action of $G$ on the associated sheaf $F$.
Take $g \in G$, and choose its open compact neighborhood $U$ such that $s$ and $r$ restrict to homeomorphisms on $U$.
Then the action of $G$ induces an isomorphism $A_{s(U)} \to A_{r(U)}$.
In turn this induces $\chi_{s(U)} K_i(A) \to \chi_{r(U)} K_i(A)$, which can be interpreted as the action of $g$ from $\Gamma(s(U), F)$ to $\Gamma(r(U), F)$.
A routine bookkeeping shows that these maps patch up to give an action morphism $s^* F \to r^* F$ on $G$.
\end{proof}

\begin{prop}\label{prop:homhk}
When $A$ is a $G$-C$^*$-algebra, there is an isomorphism of chain complexes 
\[
(K_i(G \ltimes L^{\bullet + 1} A),\delta_\bullet) \simeq (C_c(G^{(\bullet)},\Z) \otimes_{C_c(X,\Z)} K_i(A)),\partial_\bullet)
\]
for $L = \Ind^G_X \Res^G_X$.
\end{prop}

\begin{proof}
From the equivalence of groupoids between $G \ltimes G^{(n+1)}$ and $G^{(n)}$ (where we consider $G^{(k)}$ as spaces), we have $K_i(G \ltimes L^{n+1} A) \simeq K_i(C_0(G^{(n)}) \tens{s}{C_0(X)} A)$.

Since $G^{(n)}$ is totally disconnected, we have
\begin{align*}
K_0(C_0(G^{(n)})) &\simeq C_c(G^{(n)},\Z),&
K_1(C_0(G^{(n)})) = 0.
\end{align*}
Thus, we have an isomorphism of unitary $C_c(G,\Z)$-modules
\[
K_i(C_0(G^{(n)}) \tens{s}{C_0(X)} A) \simeq C_c(G^{(n)},\Z) \otimes_{C_c(X,\Z)} K_i(A).
\]
The comparison of simplicial structures is a routine calculation.
\end{proof}

Thus, we obtain an isomorphism of homology groups
\[
H_p(K_q(G \ltimes L^{\bullet + 1} A),\delta_\bullet)\simeq H_p(G,K_q(A)).
\]

\begin{theo}\label{theo:hk-with-coeff}
Let $G$ be a second countable Hausdorff ample groupoid with torsion-free stabilizers satisfying the strong Baum--Connes conjecture, and $A$ be a separable $G$-C$^*$-algebra.
Then there is a convergent spectral sequence
\begin{equation}\label{eq:spec-sec-with-coeff}
E^r_{p q} \Rightarrow K_{p+q}(G \ltimes A),
\end{equation}
with $E^2_{pq}=H_p(G,K_q(A))$.
\end{theo}

\begin{proof}
We obtain the convergent spectral sequence by Corollary \ref{cor:specseq}, and Proposition \ref{prop:homhk} gives the description of the $E^2$-sheet.
\end{proof}

Specializing to the case $A = C_0(X)$, we obtain our main result.

\begin{coro}\label{cor:hk}
Let $G$ be as above.
Then there is a convergent spectral sequence
\begin{equation*}
E^r_{p q} \Rightarrow K_{p+q}(C^*_r(G)),
\end{equation*}
with $E^2_{pq} = E^3_{pq} =H_p(G,K_q(\bC))$.
\end{coro}

\begin{proof}
As $K_q(\bC) = 0$ for odd $q$, by degree reasons the $E^2$-differential is trivial.
This implies $E^2_{pq} = E^3_{pq}$.
\end{proof}

\begin{rema}
Looking at the bidegree of differentials at the $E^3$-sheet, we see that the above spectral sequence collapses at the $E^2$-sheet if $H_k(G, \Z)$ vanishes for $k \ge 3$.
If, in addition, $H_2(G, \Z)$ is torsion-free, one has
\begin{align}\label{eq:col}
K_0(C^*_r G) &\simeq H_0(G, \Z) \oplus H_2(G, \Z),&
K_1(C^*_r G) &\simeq H_1(G, \Z).
\end{align}
This covers the transformation groupoids of minimal $\Z$-actions on the Cantor space considered in \cite{matui:hk} and the Deaconu--Renault groupoids of rank $1$ and $2$ (in particular $k$-graph groupoids for $k=1,2$) in \cite{simsfarsi:hk}, and groupoids of $1$-dimensional generalized solenoids \cite{yi:hkonesol}.
The Exel--Pardo groupoid model \cite{MR3581326} for Katsura's realization \cite{MR2400990} of Kirchberg algebras also belong to this class \cite{ortega:kep}.
For the groupoid of tiling spaces (see Section \ref{sec:subst-tiling}) one can do slightly better; if $G$ is a groupoid associated with some tiling in $\R^d$, one has the vanishing of $H_k(G, \Z)$ for $k > d$ and $H_d(G, \Z)$ is torsion-free.
Comparing the rank of $H_\bullet(G, \Z)$ and $K_\bullet(C^*_r G)$, we see that the higher differentials are always zero on $H_d(G, \Z)$, and the spectral sequence collapses if $d \le 3$.
\end{rema}

\begin{rema}
For the transformation groupoids $\Gamma \ltimes X$ where $X= E \Gamma$ is a ``nice'' manifold (such as carrying an invariant Riemannian metric with nonpositive sectional curvature), \cite{kas:descent} gives a spectral sequence analogous to \eqref{eq:spec-sec-with-coeff}.
\end{rema}

\begin{rema}
In Theorem \ref{theo:hk-with-coeff}, without assuming that $G$ has torsion-free stabilizers, or that it satisfies the strong Baum--Connes conjecture, we still have a convergent spectral sequence
\[
E^2_{pq} = H_p(G,K_q(A)) \Rightarrow K_{p+q}(G \ltimes P_A)
\]
where $P_A$ is a $(\ker \Res^G_X)$-simplicial approximation of $A$, see Section \ref{sec:non-example} for an example.
\end{rema}

\subsection{Semidirect product by torsion-free groups}
\label{sec:cr-prod-tors-free-grp}

Suppose that a group $\Gamma$ acts on a (second countable locally compact Hausdorff) groupoid $G$.
Then we can form a semidirect product $\Gamma \ltimes G$: its object set is the same as that of $G$, its arrow set is the direct product $\Gamma \times G$, with structure maps $s(\gamma, g) = s(g)$, $r(\gamma, g) = \gamma r(g)$, and composition rule $(\gamma, g) (\gamma', g') = (\gamma \gamma', \gamma'^{-1}(g) g')$.
We then have a following analogue of the permanence property of the strong Baum--Connes conjecture for extension of torsion free discrete groups \cite{MR1817505}.

\begin{prop}\label{prop:tfs}
Suppose that $\Gamma$ is torsion-free and satisfies the strong Baum--Connes conjecture, and that $G$ is an ample groupoid with torsion-free stabilizers satisfying the strong Baum--Connes conjecture.
Then any separable $\Gamma \ltimes G$-C$^*$-algebra $A$ belongs to the localizing subcategory generated by the image of $\Ind^{\Gamma \ltimes G}_X \colon \KKK^X \to \KKK^{\Gamma \ltimes G}$.
\end{prop}

\begin{proof}
Let us fix $A$ as in the assertion.
First consider the functor
\[
F\colon \KKK^\Gamma \to \KKK^{\Gamma \ltimes G}, \quad B \mapsto B \otimes A,
\]
where $\Gamma$ acts on $B \otimes A$ diagonally and $G$ acts on the leg of $A$.
This is a triangulated functor compatible with countable direct sums.

By assumption on $\Gamma$, the trivial $\Gamma$-C$^*$-algebra $\bC$ belongs to the localizing subcategory generated by objects of the form $C_0(\Gamma) \otimes B'$ for separable C$^*$-algebras $B'$.
Thus, $A = F(\bC)$ belongs to the localizing subcategory generated by the $C_0(\Gamma) \otimes B' \otimes A$.

Now, we claim that the $\Gamma \ltimes G$-C$^*$-algebra $C_0(\Gamma) \otimes A$ is isomorphic to $\Ind^{\Gamma \ltimes G}_G \Res^{\Gamma \ltimes G}_G A$, by an analogue of Fell's absorption principle.
Both algebras can be interpreted as the direct sum of copies of $A$ indexed by the elements of $\Gamma$.
For $C_0(\Gamma) \otimes A$, the action of $G$ becomes component-wise action on this direct sum, while the action of $\Gamma$ is the combination of translation on indexes and component-wise action.
For $\Ind^{\Gamma \ltimes G}_G \Res^{\Gamma \ltimes G}_G A$, the action of $G$ preserves direct summands, but twisted by the effect of $\gamma$ on $G$ on the $\gamma$-th component. The action of $\Gamma$ simply becomes translation of indexes.
We can move from one presentation to another by applying $\gamma$ or $\gamma^{-1}$ on the $\gamma$-th component.

We thus have $A$ in the localizing subcategory generated by $\Ind^{\Gamma \ltimes G}_G \Res^{\Gamma \ltimes G}_G A \otimes B'$ for separable C$^*$-algebras $B'$.
By Theorem \ref{thm:KK-induction-from-base}, $\Res^{\Gamma \ltimes G}_G A \in \KKK^G$ belongs to the localizing subcategory generated by the image of $\Ind^G_X\colon \KKK^X \to \KKK^G$.
Combined with natural isomorphism $\Ind^{\Gamma \ltimes G}_G \Ind^G_X \simeq \Ind^{\Gamma \ltimes G}_X$, we obtain the assertion (note that the C$^*$-algebras $B'$ above receive trivial action).
\end{proof}

Consequently, if $G$ is moreover ample, the conclusion of Theorem \ref{theo:hk-with-coeff} holds for $\Gamma \ltimes G$.

\begin{rema}[added after publication]
Suppose $A$ is a $\Gamma \ltimes G$-C$^*$-algebra.
We can take an $(\ker \Res^G_X)$-projective resolution $P_\bullet$ of $A$ in $\KKK^G$, such that $K_i(G \ltimes P_\bullet)$ becomes a complex of $\Gamma$-modules mapping to the $\Gamma$-module $K_i(G \ltimes A)$.
The $E^2$-sheet of the associated spectral sequence converging to $K_\bullet(G \ltimes A)$ given by Theorem \ref{theo:hk-with-coeff} is isomorphic to the hyperhomology groups $\mathbb{H}_p(\Gamma, K_q(G \ltimes P_\bullet))$, see \cite{arXiv:2104.10938}*{Section 5.1}.
\end{rema}

\section{Examples}
\label{sec:examples}

\subsection{Deaconu--Renault groupoids}

Let us sketch what one gets for the Deaconu--Renault groupoids \citelist{\cite{deaconu:endo}\cite{exren:semi}}, which reduces to the Kasparov spectral sequence for a $\Z^n$-action, see \cite{simsfarsi:hk}.

Let $X$ be a second countable totally disconnected locally compact Hausdorff space, and $\sigma$ be an action of the semigroup $\N^k$ on $X$ by surjective local homeomorphisms.
The associated \emph{Deaconu--Renault groupoid} $G = G(X, \sigma)$ is defined by
\[
G = \{(x, a-b, y) \in X \times \Z^k \times X \mid a, b \in \N^k, \sigma^a(x) = \sigma^b(y)\},
\]
with base $G^{(0)} = X$, and range and source maps given by projection onto the first and third factors.

There is a natural cocycle $c\colon G \to \Z^k$ given by $c(x, n, y) = n$, and the resulting skew-product groupoid $G \times_c \Z^k$, with base $X \times \Z^k$, range map and source maps 
\[
r((x, m, y), n) = (x, n), \qquad s((x, m, y), n) = (y, m + n).
\]
This groupoid has trivial stabilizers and is AF in the sense of \cite{simsfarsi:hk}, and in particular is a union of subgroupoids which are Morita equivalent to the space $X$.
This gives
\begin{align*}
H_0(G \times_c \Z^k, \Z) &= \varinjlim_{a \in \N^k} C_c(X, \Z),&
H_n(G \times_c \Z^k, \Z) &= 0 \quad (n > 0),
\end{align*}
where the inductive limit is taken with respect to the iteration of induced map by $\sigma$.
Moreover, by considering the automorphisms $\alpha_a \colon ((x, m, y), n) \mapsto ((x, m, y), n+a)$ for $a \in \Z^k$, we obtain a semidirect product groupoid $\tilde{G}= \Z^k \ltimes_\alpha (G \times_c \Z^k)$, which is Morita equivalent  to $G$.

Then the Leray--Hochschild--Serre spectral sequence \cite{cramo:hom}*{Theorem 4.4} applied to the canonical groupoid homomorphism $\tilde G \to \Z$ degenerates at the $E^2$-sheet and gives an isomorphism
\[
H_n(\tilde G, \Z) \simeq H_n(\Z^k, H_0(G \times_c \Z^k, \Z)),
\]
with $\Z^k$ acting on the inductive limit $\varinjlim_{a \in \N^k} C_c(X, \Z)$ by shifting the index $a$.
In particular, we get a convergent spectral sequence of the form in Corollary \ref{cor:hk} with
\[
E^r_{pq} \Rightarrow K_{p+q}(C^*_r G), \quad E^2_{p,2s} = E^3_{p,2s} = H_p(\Z^k, \varinjlim_{a \in \N^k} C_c(X, \Z)) \otimes K_q(\bC), \quad E^2_{p,{2s+1}} = E^3_{p,2s} = 0
\]
as in \cite{simsfarsi:hk}.

\subsection{Substitution tiling}
\label{sec:subst-tiling}

We follow the convention of \cite{MR1798993}, and consider substitution tilings of finite local complexity.
Thus, we are given a finite set $P$ of \emph{prototiles} in $\R^d$ and a substitution rule $\omega$ for~$P$.
Under reasonable assumptions on $\omega$, the translation action $\tau$ of $\R^d$ on the associated hull space $\Omega$ is free and minimal.
Then, analogously to the case of solenoids, the groupoid of the unstable equivalence relation is the transformation groupoid $G = \R^d \ltimes_\tau \Omega$.
Moreover, by \cite{MR1971208}, there is a transversal $X \subset \Omega$ that is homeomorphic to a Cantor set, such that $(\R^d \ltimes_\tau \Omega)|_X$ is the transformation groupoid $\Z^d \ltimes_\alpha X$ for some action $\alpha\colon \Z^d \curvearrowright X$, see also \cite{MR2021006}*{Section 5}.

Let us quickly explain how a spectral sequence of more classical nature arises in this setting.
By Connes's Thom isomorphism, one has
\[
K_n(C^* G) \simeq K^{n+d}(\Omega).
\]
Now, $\Omega$ can be identified with a projective limit of a self-map of branched $d$-dimensional manifold obtained by gluing (collared) prototiles \cite{putand:til}.
This leads to the Atiyah--Hirzebruch spectral sequence
\begin{equation}
\label{eq:grpd-cohom-sp-seq-to-k-grp}
E_2^{p,q} = {\check H}^p(\Omega, K_q(\bC)) \Rightarrow K^{p+q}(\Omega),
\end{equation}
that is, $E_2^{p,q}$ is the $p$-th Čech cohomology of $\Omega$ with constant sheaf $\underline\Z$ when $q$ is even, and  $E_2^{p,q} = 0$ otherwise (for dimension reasons we also have $E_2^{p,q} = 0$ if $p > d$).
Since $\Omega$ is a compact Hausdorff space, this is also equal to the sheaf cohomology as derived functor.
Since the action $\tau$ is free and $\R^d$ is contractible, $\Omega$ is a model of the classifying space $B G$ and the universal principal bundle $E G$ for the groupoid $G$ (up to nonequivariant homotopy).
In particular, we can interpret the sheaf cohomology on $\Omega$ as groupoid cohomology of $G$, see \citelist{\cite{MR1607724}\cite{MR2231869}}.

Let us relate our construction to this.
Using the transversal $X$, we have
\begin{align*}
H^\bullet(G|_X, \Z) &\simeq H^\bullet(\Z^d, C(X, \Z)),&
H_\bullet(G|_X, \Z) &\simeq H_\bullet(\Z^d, C(X, \Z)),&
\end{align*}
where we consider $C(X, \Z)$ as a module over $\Z^d$ by the action induced by $\alpha$.
Moreover we have $H_k(\Z^d, M) \simeq H^{d-k}(\Z^d, M)$ for any $\Z^d$-module $M$, see for example \cite{MR1324339}*{Section VIII.10}.
This shows that
\[
H_k(G|_X, \Z) \simeq H^{d-k}(G|_X, \Z) \simeq H^{d-k}(G, \Z)
\]
for the étale groupoid $G|_X$, and the spectral sequence of Corollary \ref{cor:hk} is comparable to \eqref{eq:grpd-cohom-sp-seq-to-k-grp}.

\begin{rema}
A similar spectral sequence is given in \cite{MR2505326}, as an analogue of the Serre spectral sequence for the Anderson--Putnam fibration structure $\Omega \to \Gamma_k$ over the $k$-collared prototile space.
It would be an interesting question to compare these.
\end{rema}

\subsection{A non-example}
\label{sec:non-example}

Scarparo has found a counterexample to the HK conjecture \cite{eduardo:hk}.
In his example $G$ is the transformation groupoid of an action $\alpha$ of the infinite dihedral group $\Gamma = \Z_2 \ltimes \Z$ on the Cantor set $X$.
Thus, it is amenable and in particular satisfies the strong Baum--Connes conjecture.
However, the simplicial approximation $P(C(X))$ arising from restriction to the unit space is indeed not $\KKK^G$-equivalent to $C(X)$.
Let us explain the ingredients in more detail.

Let $(n_i)_{i=0}^\infty$ be a strictly increasing sequence of integers such that, for $i\geq 1$, $n_{i+1} / n_i \in \N$ for all~$i$.
We take the model $X=\varprojlim \Z_{n_i}$.
Then $\Z$ acts by the odometer action, i.e., $1 \in \Z$ acts by the $+1$ map on each factor $\Z_{n_i}$.
There is a consistent action of $\Z_2$, where the nontrivial element $g = [1] \in \Z_2$ acts by multiplication by $-1$, giving rise to an action $\alpha$ of $\Gamma$ on $X$.
Note that $\alpha$ is topologically free but not free, nor does it have torsion-free stabilizers.

Put $G = \Gamma \ltimes_\alpha X$, and
\[
M = \left\{\frac{m}{n_i}\mid m\in\Z,i\geq 1\right\}.
\]
The C$^*$-algebra $C^* G = \Gamma \ltimes_\alpha C(X)$ is an AF algebra, with
\[
K_0(C^* G) \simeq \begin{cases}
M\oplus \Z & \text{if $n_{i+1}/ n_i$ is even for infinitely many $i$,}\\
M \oplus \Z^2 & \text{otherwise,}
\end{cases}
\]
see \cite{MR1245825}.
On the other hand, the groupoid homology is
\begin{align*}
H_0(G,\Z) &\simeq M,\\
H_{2k}(G,\Z) &\simeq 0,\\
H_{2k-1}(G,\Z) &\simeq\begin{cases}
\Z_{2}& \text{if $n_{i+1}/ n_i$ is even for infinitely many $i$,}\\
\Z_{2}^2 & \text{otherwise,}
\end{cases}
\end{align*}
for $k > 1$, see \cite{eduardo:hk}.
This shows that groupoid homology cannot form a spectral sequence converging to $K_\bullet(C^* G)$, much less being isomorphic to it.

Fortunately, there is a somewhat concrete description of $P(C(X))$ in this case.
Consider the antipodal action of $\Z_2$ on $S^n$, that is, $g$ acts by the restriction of the multiplication by $-1$ on $\R^{n+1}$.
Then the contractible space $S^\infty = \varinjlim S^n$ is a model of the universal bundle $E \Z_2$.
We want to make sense of an analogue of Poincaré dual for this.

Let $Y_n = C_0(T^* S^n)$ denote the function algebra of the total space of the cotangent bundle of $S^n$, and $Y'_n$ denote the $\Z_2$-graded C$^*$-algebra of continuous sections of the C$^*$-algebra bundle $\Cliff(T^* S^n)$ over $S^n$ with complex Clifford algebras $\Cliff(T^*_x S^n)$ as fibers.
These admit naturally induced actions of $\Z_2$, and $Y_n$ is $\KKK^{\Z_2}$-equivalent to $Y_n'$ \cite{MR3597146}*{Theorem 2.7}.

Let us recall the (equivariant) Poincaré duality between $C(S^n)$ and $Y'_n$ \cite{kas:descent}*{Section 4}.
The natural Clifford module structure on the differential forms of $S^n$, together with $D'_n = d + d^*$, give an unbounded model of a $K$-homology class $[D'_n] \in K^0_{\Z_2}(Y'_n)$.
Composed with the product map $m \colon Y'_n \otimes C(S^n) \to Y'_n$, we obtain the class $[D_n] = m \otimes_{Y'_n} [D'_n] \in K^0_{\Z_2}(Y'_n \otimes C(S^n))$.
The dual class $[\Theta_n] \in K_0^{\Z_2}(C(S^n) \otimes Y'_n)$ is defined as a certain class localized around the diagonal.

Let $j\colon S^n \to S^{n+1}$ be the embedding at the equator (which is a $\Z_2$-equivariant continuous map), and let $j' \colon Y'_n \to Y'_{n+1}$ be the $\KKK^{\Z_2}$-morphism dual to the restriction map $j^*\colon C(S^{n+1}) \to C(S^n)$.
Thus, we have
\[
j' = [\Theta_{n+1}] \otimes_{C(S^{n+1}) \otimes Y_{n+1}} (\id_{Y'_n} \otimes j^* \otimes \id_{Y'_n}) \otimes_{Y_n \otimes\, C(S^n)} [D_n],
\]
see \cite{kas:descent}*{Theorem 4.10}.

\begin{lemm}
We have $j' \otimes_{Y'_{n+1}} [D'_{n+1}] = [D'_n]$ in $K^0_{\Z_2}(Y'_n)$.
\end{lemm}

\begin{proof}
As a $\KKK^{\Z_2}$-morphism, $[D'_n]$ is the dual of the embedding $\eta_n \colon \bC \to C(S^n)$, hence the claim reduces to $\eta_{n+1} = j \eta_n$.
\end{proof}

Take the homotopy colimit $Y'_\infty = \varinjlim Y'_n$ in $\KKK^{\Z_2}$ (to be precise, we are working in the enlarged category of $\Z_2$-graded C$^*$-algebras).
By the above lemma, the morphisms $[D'_n]$ induce a morphism $[D'_\infty] \in \KKK^{\Z_2}(Y'_\infty, \bC)$.
Transporting this by the $\KKK^{\Z_2}$-equivalence, we obtain $Y_\infty = \varinjlim Y_n$ and $[D_\infty] \in \KKK^{\Z_2}(Y_\infty, \bC)$.

\begin{lemm}
The image of $[D_\infty]$ in $\KKK(Y_\infty, \bC)$ is a $\KKK$-equivalence.
\end{lemm}

\begin{proof}
In the nonequivariant $KK$-category, $Y_n$ is equivalent to $\bC^2$ or $\bC \oplus \Sigma \bC$ depending on the parity of $n$, and there is a distinguished summand which is equivalent to $\bC$ (at the even degree) spanned by the $K$-theoretic fundamental class of $T^* S^n$.
Moreover, the morphism corresponding to $[D'_n]$ is a projection onto this summand.

The $\KKK$-morphisms corresponding to $j'$ preserve the fundamental class while killing the other direct summand.
Thus, the limit is equivalent to $\bC$, spanned by the image of the fundamental classes, and $[D_\infty]$ gives the equivalence.
\end{proof}

Since $\Z_2$ acts freely on $T^* S^n$, each $Y_n$ is orthogonal to the kernel of restriction functor $\KKK^{\Z_2} \to \KKK$.
The discussion so far can be readily adjusted to the groupoid $G$, as follows.
Here, $Y_n \otimes C(X)$ is a $G$-C$^*$-algebra for which $Y_n$ only sees the action of $\Z_2$.

\begin{prop}
The $G$-C$^*$-algebra $Y_n \otimes C(X)$ belongs to the localizing subcategory generated by the image of $\Ind^G_X \colon \KKK^X \to \KKK^G$.
\end{prop}

\begin{proof}
First, $G \ltimes (T^* S^n \times X)$ is a free groupoid.
Indeed, it is the transformation groupoid of the action $\Gamma \curvearrowright T^* S^n \times X$, but any element $\gamma \in \Gamma$ that has a fixed point in $X$ is either conjugate to $(g, 0)$ or $(g, 1)$.
(Here, $g$ is the nontrivial element of $\Z_2$ and we identify $\Gamma$ with $\Z_2 \times \Z$ as a set.)
By the freeness of $\Z_2 \curvearrowright T^* S^n$, these elements cannot have fixed points in $T^* S^n \times X$.

We thus obtain that $Y_n \otimes C(X)$ belongs to the localizing subcategory generated by the image of $\Ind^{G \ltimes (T^* S^n \times X)}_{T^* S^n \times X}$, see the proof of Theorem \ref{thm:KK-induction-from-base}.
Using the triangulated functor $\KKK^{G \ltimes (T^* S^n \times X)} \to \KKK^G$ given by restricting the scalars of $C_0(T^* S^n \times X)$-algebras to $C(X)$, we obtain the assertion.
\end{proof}

\begin{coro}
We have $P_\cI C(X) \simeq Y_\infty \otimes C(X)$ for $\cI = \ker \Res^G_X$, with the corresponding $\KKK^G$-morphism $Y_\infty \otimes C(X) \to C(X)$ given by $[D_\infty] \otimes \id_{C(X)}$.
\end{coro}

Consequently, the spectral sequence of groupoid homology converges to the $K$-theory groups of the algebra $G \ltimes (Y_\infty \otimes C(X))$.

\appendix
\section{Structure of groupoid equivariant \texorpdfstring{$\KKK$}{KK}-theory}
\label{sec:str-grpd-equiv-KK}

As in the other parts of paper, $G$ denotes a locally compact Hausdorff groupoid with Haar system, and we write $X = G^{(0)}$.
We denote the category of separable $G$-$C^*$-algebras by $\GCalg$.
We regard $C_c(G)$ as a $C_0(X)$-module via pullback by $s$, and denote its completion a right Hilbert $C_0(X)$-module with respect to the inner product by the Haar system by $L^2(G)$.

\subsection{Invariant ideals}

Let us check that continuous actions of $G$ restrict to kernels of equivariant homomorphisms.

\begin{prop}
Let $f\colon A \to B$ be an equivariant homomorphism of $G$-C$^*$-algebras.
Then $I = \ker f$ is a $G$-C$^*$-algebra.
\end{prop}

\begin{proof}
Since $I$ is an ideal of $A$, it inherits a structure of $C_0(X)$-algebra.
We need to show that there is an isomorphism of $C_0(G)$-algebras
\[
s^* I = C_0(G) \tens{s}{C_0(X)} I \to r^* I = C_0(G) \tens{r}{C_0(X)} I
\]
defining a continuous action of $G$.
By the nuclearity of $C_0(G)$ as a C$^*$-algebra,
\[
0 \to C_0(G) \otimes I \to C_0(G) \otimes A \to C_0(G) \otimes B \to 0
\]
is exact.

We first claim that $s^* I$ is the kernel of $s^* A \to s^* B$ induced by $f$.
By the $C_0(X)$-nuclearity of $C_0(G)$, we can write
\[
s^* I = (C_0(G) \otimes I)_{\Delta(X)},
\]
etc. Then we have a commutative diagram
\[
\begin{tikzcd}
& 0 \arrow[d] & 0 \arrow[d] & 0 \arrow[d] &\\
0 \arrow[r] & I' \arrow[r] \arrow[d] & A' \arrow[r] \arrow[d] & B' \arrow[r] \arrow[d] & 0\\
0 \arrow[r] & C_0(G) \otimes I \arrow[r] \arrow[d] & C_0(G) \otimes A \arrow[r] \arrow[d] & C_0(G) \otimes B \arrow[r] \arrow[d] & 0\\
0 \arrow[r] & s^* I \arrow[r] \arrow[d] & s^* A \arrow[r] \arrow[d] & s^* B \arrow[r] \arrow[d] & 0\\
& 0 & 0 & 0 &
\end{tikzcd}
\]
with $I' = C_0((G \times X) \setminus (G \times_X X)) (C_0(G) \otimes I)$, etc., and we know the exactness of the vertical sequences and top and middle horizontal sequences.
Then the bottom sequence is also exact, which establishes the claim.

Then looking at the action map
\[
s^* A \to r^* A,
\]
we see that $s^* I$ is mapped onto $r^* I = \ker(r^* A \to r^* B)$ bijectively.
\end{proof}

\subsection{Stabilization}

When $\cE$ is a Hilbert $A$-module, we denote the Hilbert $A$-module direct sum of countable copies of $\cE$ by $\cE^\infty$.
From now on let us assume that $G$ admits a Haar system $\lambda$, so that $L^2(G)$ makes sense as a Hilbert $C_0(X)$-module.

\begin{lemm}\label{lem:Hilb-mod-stab-by-L2G-infty}
Let $A$ be a separable $G$-C$^*$-algebra, and $\cE$ be a countably generated Hilbert $G$-$A$-module.
If $E$ is full as an right Hilbert $A$-module, we have
\[
L^2(G)^\infty \otimes_{C_0(X)} \cE \simeq L^2(G)^\infty \otimes_{C_0(X)} A
\]
as Hilbert $G$-$A$-modules.
\end{lemm}

\begin{proof}
By fullness, we have $\cE^\infty \simeq A^\infty$ as Hilbert $A$-modules \cite{MR1325694}*{Proposition 7.4}.
Then the assertion follows from \cite{MR2044224}*{Lemma 3.6}.
\end{proof}

\begin{prop}\label{prop:stability-conditions-equiv}
Let $F$ be a functor from $\GCalg$ to an additive category.
Then the following conditions are equivalent:
\begin{enumerate}
\item\label{it:st-1} if $\cE$ is a Hilbert $G$-$A$-module which is full over $A$, the natural maps
\begin{align*}
F(A) &\to F(\cK(A \oplus \cE)),&
F(\cK(\cE)) &\to F(\cK(A \oplus \cE))
\end{align*}
are isomorphisms;
\item\label{it:st-2} same as above, but just for $\cE = \cE' \otimes_{C_0(X)} A$ where $\cE'$ is a Hilbert $G$-$C_0(X)$-module which is full over $C_0(X)$;
\item\label{it:st-3} same as above, but just for $\cE' = L^2(G)^\infty$.
\end{enumerate}
\end{prop}

\begin{proof}
The only nontrivial implication is from \eqref{it:st-3} to \eqref{it:st-1}.
Since $\cK(L^2(G)^\infty \otimes_{C_0(X)} A)$ is isomorphic to $\cK(L^2(G)^\infty) \otimes_{C_0(X)} A$, \eqref{it:st-3} implies that $F(A) \simeq F(\cK(L^2(G)^\infty)\otimes_{C_0(X)}A)$.
Suppose $\cE$ is as in \eqref{it:st-1}.
Then $\cK(L^2(G)^\infty) \otimes_{C_0(X)} \cK(\cE)$ is isomorphic to $\cK(L^2(G)^\infty \otimes_{C_0(X)} A)$ by this observation and Lemma \ref{lem:Hilb-mod-stab-by-L2G-infty}.
We thus obtain $F(A) \simeq F(\cK(\cE))$, and a routine bookkeeping gives that this can be indeed induced by maps as in \eqref{it:st-1}.
\end{proof}

If the conditions in the above proposition are satisfied, we say that $F$ is \emph{stable}.

\subsection{Universal property}
\label{sec:univ-prop-KKG}

Again suppose $F$ is a functor from $\GCalg$ to an additive category.
As usual, $F$ is \emph{homotopy invariant} if the evaluation maps $A \otimes C([0, 1]) \to A$ at $0 \le t \le 1$ induce isomorphisms $F(A \otimes C([0, 1])) \simeq F(A)$, and is \emph{split exact} if an extension $I \to A \to B$ with splitting $B \to A$ by an equivariant $*$-homomorphism induces an isomorphism $F(A) \simeq F(I) \oplus F(B)$.

\begin{prop}\label{prop:univ-prop-KKG}
The canonical functor $\GCalg \to \KKK^G$ is a universal functor satisfying stability, homotopy invariance, and split exactness. 
\end{prop}

Before getting into the proof, recall that an element in $\KKK^G(A, B)$ is by definition represented by a $G$-$A$-$B$-Kasparov cycle $(\pi, \cE, T)$, where $\cE$ is a $\Z_2$-graded right Hilbert $B$-module, $\pi$ is a $*$-homomorphism from $A$ to $\cL(\cE)$, and $T$ is a certain odd endomorphism of $\cE$.
Note that $T$ is only assumed to be $G$-equivariant up to compact errors.
A key ingredient is the following result of Oyono-Oyono, which allows us to replace such cycles by strictly equivariant ones.
(To be precise, his result is for odd cycles, but his construction is compatible with grading on the underlying Hilbert module, otherwise we can work with suspensions.)

\begin{prop}[\cite{laff:kkban}*{Section A.4}]\label{prop:T-equivar-up-to-compact}
Under the above setting, there is an odd $G$-equivariant endomorphism $T'$ on $\tilde\cE = L^2(G)^\infty \otimes_{C_0(X)} \cE$ such that $(\iota \otimes \pi, \tilde\cE, T')$ is a $G$-$(\cK(L^2(G)^\infty) \otimes_{C_0(X)} A)$-$B$-Kasparov cycle, and such that $(S \otimes \pi(a)) (1 \otimes T - T')$ is a compact endomorphism for all $S \in \cK(L^2(G)^\infty)$ and $a \in A$.
\end{prop}

Another important ingredient is the ``Cuntz picture'' of $\KKK^G(A, B)$.
To simplify the notation, put $\tA = \cK(L^2(G)^\infty) \otimes_{C_0(X)} A$.
A \emph{$G$-equivariant quasi-homomorphism} from $\tA$ to $\tB$ is given by pair of $G$-equivariant $*$-homomorphisms $\phi^+, \phi^-$ from $\tA$ to $\cM(\tB)$ such that $\phi^+(a) - \phi^-(a) \in \tB$ for all $a \in \tA$.
This induces a Kasparov $G$-$\tA$-$B$-cycle
\begin{equation}\label{eq:kasparov-cycle-from-quasi-hom}
\left(\phi^+ \oplus \phi^-, (L^2(G)^\infty \otimes_{C_0(X)} B)^{\oplus 2}, T = 
\begin{bmatrix}
 0 & 1\\
 1 & 0
\end{bmatrix}\right).
\end{equation}

\begin{proof}[Proof of Proposition \ref{prop:univ-prop-KKG}]
Let us explain how to present $\KKK^G(A, B)$ in terms of equivariant quasi-homomorphisms using Proposition \ref{prop:T-equivar-up-to-compact}.
Let us start with a $G$-$A$-$B$-Kasparov cycle $(\pi, \cE, T)$.
Adding a degenerate direct summand, we may assume that $E$ is full as a Hilbert $B$-module.
Take a $G$-equivariant endomorphism $T'$ on $\tilde\cE$ as above.
$1 \otimes T$ and $T'$ define homotopic Kasparov cycles, so any class in $\KKK^G(A, B)$ has a $G$-equivariant representative by replacing $A$ with $\tA$.

Doing the same for $G$-$A$-$(B \otimes C([0, 1]))$-Kasparov cycles, we see that, if $(\pi_0, \cE_0, T_0)$ and $(\pi_1, \cE_1, T_1)$ are homotopic cycles, then $T'_0$ are $T'_1$ are homotopic through a $G$-equivariant path.
Consequently $\KKK^G(A, B)$ is the quotient set of the $G$-$\tA$-$B$-Kasparov cycles $(\pi, \cE, T)$, with $G$-equivariant $T$, up to the equivalence relation generated by $G$-equivalent homotopy, $G$-equivariant unitary equivalence, and ignoring the difference of direct sum with degenerate cycles.

Moreover, we can replace $T$ by a $G$-equivariant endomorphism satisfying $T = T^* = T^{-1}$ without breaking the equivalence relation, see \cite{MR859867}*{Chapter 17}.
By Lemma \ref{lem:Hilb-mod-stab-by-L2G-infty}, we may assume that $T$ is represented on $L^2(G)^\infty \otimes_{C_0(X)} B$.
Then we can write $T$ in the form of \eqref{eq:kasparov-cycle-from-quasi-hom}, and the left action of $\tA$ is given by a $G$-equivariant $*$-homomorphism $\pi\colon \tA \to \cM(\tB)$.
Finally, the commutation relation with $T$ implies that $\pi$ is of the form $\phi^+ \oplus \phi^-$ for an equivariant quasi-homomorphism $(\phi^+, \phi^-)$ from $\tA$ to $\tB$.
Consequently, $\KKK^G(A, B)$ is isomorphic to the set of equivalence classes of equivariant quasi-homomorphisms $(\phi^+, \phi^-)\colon \tA \to \tB$ up to equivariant homotopy, equivariant unitary equivalence, and ignoring the difference of direct sum with degenerate  ones.

Next let us relate quasi-homomorphisms to split extensions, cf.~\citelist{\cite{MR859867}\cite{MR1789948}}.
Let $(\phi^+, \phi^-)$ be an equivariant quasi-homomorphism from $\tA$ to $\tB$.
Put
\[
D = \{(a, \phi^+(a) + b) \mid a \in \tA, b \in \tB\} \subset \tA \oplus \cM(\tB),
\]
which is a $G$-C$^*$-algebra by Lemma \ref{lem:D-is-G-C-alg}.
Moreover, this fits into a split extension
\[
\begin{tikzcd}
\tB \arrow{r}{j} & D \arrow[shift left]{r}{q} & \tA \arrow[shift left]{l}{s}, 
\end{tikzcd}
\]
with $j(b) = (0, b)$, $q(a, \phi^+(a) + b)) = a$, and $s(a) = (a, \phi^-(a))$.

Suppose that $F\colon \GCalg \to \cC$ is a functor into an additive category satisfying stability, homotopy invariance, and split exactness.
We want to show that there is a uniquely determined functor $\tilde F \colon \KKK^G \to \cC$ factoring $F$ up to natural isomorphisms.
Given an equivariant quasi-homomorphism $(\phi^+, \phi^-)$ from $\tA$ to $\tB$, construct $D$ as above.
Then we have an identification $F(D) \simeq F(\tB) \oplus F(\tA)$, so the projection onto the first summand combined with stability gives a morphism $\phi_*\colon F(D) \to F(B)$.
Moreover, there is another equivariant $*$-homomorphism $f\colon \tA \to D$ defined by $f(a) = (a, \phi^+(a))$.
We then obtain $\tilde F(\phi^+, \phi^-)\colon F(A) \to F(B)$ by combining $\phi_* \circ F(f)$ with stability for $A$.
This construction is compatible with the equivalence relation on quasi-homomorphisms, and we obtain a well-defined functor $\tilde F\colon \KKK^G \to \cC$ extending $F$.

Uniqueness follows from functoriality and the following observation: if $B \to D \to A$  is a split extension, $D$ is a model for the direct sum $B \oplus A$ in $\KKK^G$.
More concretely, the ideal inclusion $j\colon B \to D$ defines a homomorphism $\tilde \jmath\colon D \to \cM(B)$, and $(\tilde \jmath, \tilde \jmath s q)$ in the above notation defines a quasi-homomorphism from $D$ to $B$.
This is a projection of $D$ to $B$ in $\KKK^G$, and together with the other maps in the extension these $\KKK^G$-morphisms give the structure morphisms of the direct~sum.
\end{proof}

\begin{lemm}
\label{lem:D-is-G-C-alg}
The algebra $D$ in the above proof is a $G$-C$^*$-algebra.
\end{lemm}

\begin{proof}
First let us check that $D$ is a $C_0(X)$-algebra.
If $f \in C_0(X)$, $a \in \tA$ and $b \in \tB$, we obviously put $f . (a, \phi^+(a) + b) = (f a, \phi^+(f a) + f b)$.
Since $f a$ and $f b$ can approximate $a$ and $b$, we see that this defines a nondegenerate homomorphism $C_0(X) \to \cM(D)$.

Next, the maps
\begin{align*}
D &\to \tA, \quad (a, \phi^+(a) + b) \mapsto a,&
D &\to \tB, \quad (a, \phi^+(a) + b) \mapsto b
\end{align*}
are $C_0(X)$-linear and completely bounded.
From this we see that $C_0(G) \tens{s}{C_0(X)} D$ is a direct sum of $C_0(G) \tens{s}{C_0(X)} \tA$ and $C_0(G) \tens{s}{C_0(X)} \tB$ as an operator space, and similar decomposition holds for $C_0(G) \tens{r}{C_0(X)} D$.
Then we obtain an action map on $D$ as combination of the action maps on $\tA$ and $\tB$.
\end{proof}

\subsection{Triangulated structure}
\label{sec:app-triagl-str-equiv-KK}

Let $f \colon A \to B$ be an equivariant $*$-homomorphism of $G$-C$^*$-algebras.
As usual, its mapping cone is given by
\[
\Cone(f) = \{(a, b_*) \in A \oplus C_0((0, 1], B) \mid f(a) = b_1 \},
\]
which inherits a structure of $G$-C$^*$-algebra from $A$ and $B$.

An exact triangle in $\KKK^G$ is a diagram of the form
\[
A \to B \to C \to \Sigma A
\]
such that there exists a homomorphism $f\colon A' \to B'$ of $G$-C$^*$-algebras and a commutative diagram
\[
\begin{tikzcd}
A \arrow[r] \arrow[d] & B \arrow[r] \arrow[d] & C \arrow[r] \arrow[d] & \Sigma A  \arrow[d]\\
\Sigma B' \arrow[r] & \Cone(f) \arrow[r] & A' \arrow[r] & B',
\end{tikzcd}
\]
in $\KKK^G$, where vertical arrows are equivalences and the rightmost downward arrow is equal to the leftmost downward arrow up to applying $\Sigma$ and Bott periodicity $\Sigma^2 B' \simeq B'$ in $\KKK^G$.

Thus, we are really defining a triangulated category structure on the opposite category of $\KKK^G$.
Generally the opposite category of a triangulated category is again triangulated with ``the same'' exact triangles with suspension and desuspension exchanged, but for $\KKK^G$ we have $\Sigma^2 \simeq \Id$ and we can ignore that issue.

The crucial step is to check the axiom (TR1), in particular that any $\KKK^G$-morphism is represented by a $G$-equivariant $*$-homomorphism up to $\KKK^G$-equivalence, see \cite{laff:kkban}*{Lemma A.3.2}.
Having established that, the rest is quite standard; one can follow \cite{nestmeyer:loc}*{Appendix A} to check that the triangles of the form
\[
\Sigma B \to \Cone(f) \to A \to B
\]
satisfy the axioms (TR2), (TR3), and (TR4) for the opposite category of $\KKK^G$.

Finally, suppose that an equivariant $*$-homomorphism $f\colon A \to B$ is surjective with a $C_0(X)$-linear completely positive section $B \to A$.
Then the $G$-C$^*$-algebra $I = \ker f$ is isomorphic to $\Cone(f)$ in $\KKK^G$, by the embedding homomorphism
\[
I \to \Cone(f), \quad a \mapsto (a, 0).
\]
It follows that there is an exact triangle of the form
\[
\begin{tikzcd}
I \arrow[r] & A  \arrow[r,"f"] &  B \arrow[r] & \Sigma I.
\end{tikzcd}
\]

\raggedright

\end{document}